\documentclass[a4paper,12pt]{article}

\usepackage{typearea}
\typearea{15}

\usepackage{mymacros}

\title{Periods of tropical Calabi--Yau hypersurfaces}
\author{Yuto Yamamoto}
\date{}
\begin{document}

\maketitle

\begin{abstract}
We consider the residual B-model variation of Hodge structure of Iritani defined by a family of toric Calabi--Yau hypersurfaces over a punctured disk $D \setminus \lc 0\rc$.
It is naturally extended to a logarithmic variation of polarized Hodge structure of Kato--Usui on the whole disk $D$.
By restricting it to the origin, we obtain a polarized logarithmic Hodge structure (PLH) on the standard log point.
In this paper, we describe the PLH in terms of the integral affine structure of the dual intersection complex of the toric degeneration in the Gross--Siebert program.

\begin{comment}
Suppose that we are given a family of K\"{a}hler manifolds $\lc V_q \rc_{q \in D \setminus \lc 0 \rc}$ over a punctured disk $D \setminus \lc 0 \rc$ whose monodromy is unipotent.
The variation of polarized Hodge structure over the punctured disk defined by the family is naturally extended to a logarithmic variation of polarized Hodge structure of Kato--Usui on the whole disk $D$.
By restricting it to the origin $0 \in D$, we can get a polarized logarithmic Hodge structure (PLH) on the standard log point.
We expect that this PLH is read off from the tropicalization of the original family $\lc V_q \rc_{q \in D}$.
In the talk, we describe the PLH in terms of the integral affine structure of the dual intersection complex of the toric degeneration in the Gross--Siebert program in the case of families of toric Calabi--Yau hypersurfaces and residual B-model variation of Hodge structure.
\end{comment}

\begin{comment}
With a family of toric Calabi--Yau hypersurfaces over a punctured disk, one can associate a logarithmic variation of polarized Hodge structure (LVPH) on the standard log point in two different ways.
One is the natural extension to the origin of the residual B-model variation of Hodge structure in the sense of Iritani, and the other is constructed using the radiance obstruction of the integral affine sphere with singularities obtained from the family by tropicalization.
The main result in this paper shows that the resulting LVPHs are isomorphic.
\end{comment}
\end{abstract}

%--------------------------------
\section{Introduction}\label{sc:intro}
%--------------------------------

Suppose that we are given a family of K\"{a}hler manifolds $\lc V_q \rc_{q \in D_\varepsilon \setminus \lc 0 \rc}$ over a punctured disk $D_\varepsilon \setminus \lc 0 \rc$ whose monodromy is unipotent.
The variation of polarized Hodge structure over the punctured disk defined by the family is naturally extended to a logarithmic variation of polarized Hodge structure (LVPH) of Kato--Usui \cite{MR2465224} on the whole disk $D_\varepsilon$.
By restricting it to the origin $0 \in D_\varepsilon$, we can get a polarized logarithmic Hodge structure (PLH) on the standard log point.
We expect that this PLH is read off from the tropicalization of the original family $\lc V_q \rc_{q \in D_\varepsilon \setminus \lc 0\rc}$.
In this paper, we accomplish a first step in this direction for families of toric Calabi--Yau hypersurfaces and residual B-model variation of Hodge structure.

The residual B-model VHS and the ambient A-model VHS were introduced by Iritani \cite[Section 6]{MR3112512} in order to study Hodge theoretic mirror symmetry for Batyrev's mirror pairs of Calabi--Yau hypersurfaces \cite{MR1269718}.
It is known that for Batyrev's mirror pair of Calabi--Yau hypersurfaces $(Y, \check{Y}_\alpha)$, the residual B-model variation of Hodge structure of $\check{Y}_\alpha$ and the ambient A-model variation of Hodge structure of $Y$ are isomorphic including their integral structures via the mirror isomorphism \cite[Theorem 6.9]{MR3112512}.
The main theorem of this paper describes the PLH on the standard log point defined by the residual B-model variation of Hodge structure of Calabi--Yau hypersurfaces in terms of the dual intersection complex $B$ of a toric degeneration in the Gross--Siebert program \cite{MR2198802}, \cite{MR2213573}.
In order to prove the theorem, we compute the radiance obstruction of $B$, which is an invariant of integral affine manifolds introduced in \cite{MR760977}.

The precise setup of this paper is the following:
Let $d$ be a positive integer.
Let further $M$ be a free $\bZ$-module of rank $d+1$ and $N:=\Hom(M, \bZ)$ be the dual lattice.
We set $M_\bR:=M \otimes_\bZ \bR$ and $N_\bR:=N \otimes_\bZ \bR=\Hom(M, \bR)$.
Let $\Sigma \subset N_\bR, \Sigmav \subset M_\bR$ be unimodular fans 
(i.e., fans such that all of their cones are generated by a subset of a basis of $N$ and $M$ respectively)
whose fan polytopes $\Deltav \subset N_\bR, \Delta \subset M_\bR$ 
(i.e., the convex hulls of primitive generators of one-dimensional cones) are polar dual to each other.
Recall that we say that a function $\check{f} \colon M_\bR \to \bR$ is convex if it satisfies $\check{f} \lb t m_1+(1-t)m_2 \rb \leq t \check{f}(m_1)+(1-t)\check{f}(m_2)$ for any $t \in \ld 0, 1\rd$ and $m_1, m_2 \in M_\bR$.
We also say that a convex function $\check{f} \colon M_\bR \to \bR$ is \emph{strictly} convex on a complete fan in $M_\bR$ if $\check{f}$ is linear on each cone of dimension $d+1$ in the fan, and distinct cones of dimension $d+1$ correspond to distinct linear functions.
Let $\check{h} \colon M_\bR \to \bR$ be a piecewise linear function that is strictly convex on the fan $\Sigmav$, and $\varphiv \colon M_\bR \to \bR$ be the piecewise linear function corresponding to the anti-canonical sheaf on the toric varieties associated with the normal fan of $\Deltav$.
We assume that the function
$\check{h}':=\check{h}-\varphiv \colon M_\bR \to \bR$
is convex (not necessarily strictly convex) on $\Sigmav$.

We consider the $d$-sphere $B$ equipped with an integral affine structure with singularities which is constructed from the data $(\Sigma, \check{h})$ in \cite{MR2198802} as the dual intersection complex of the toric degeneration.
We will recall the construction in \pref{sc:construction}.
\begin{comment}
The $d$-sphere $B$ can also be regarded as the space which is obtained by contracting the tropical hypersurface defined by the tropical Laurent polynomial $f \colon N_\bR \to \bR$ given by
\begin{align}\label{eq:poly}
	f(n)= \min_{m \in \Delta \cap M} \lc \check{h}(m) + \la m, n \ra \rc,
\end{align}
where $\la -, - \ra$ is the pairing between  $M_\bR$ and $N_\bR$.
We will explain a little bit more about this in \pref{rm:contraction}.
\end{comment}
It is conjectured that maximally degenerating families of Calabi--Yau manifolds with Ricci-flat K\"{a}hler metrics converge to $d$-spheres with integral affine structures with singularities in the Gromov--Hausdorff topology \cite{MR1863732}, \cite{MR2181810}, and 
in the case of the toric degeneration, it is expected that it converges to the dual intersection complex $B$ of the toric degeneration.

Let $\iota \colon B_0 \hookrightarrow B$ denote the complement of singularities of $B$, and $\scT_\bZ$ be the local system on $B_0$ of lattices of integral tangent vectors.
We set $\scT_Q:=\scT_\bZ \otimes_\bZ Q$ for $Q=\bR, \bC$, and define
\begin{align}
H^\bullet_\mathrm{f} \lb B, \iota_\ast \bigwedge^\bullet \scT_\bZ \rb
&:=\bigoplus_{i=0}^d \left. H^i \lb B, \iota_\ast \bigwedge^i \scT_\bZ \rb \middle/ T_i \right. ,\\
H^\bullet \lb B, \iota_\ast \bigwedge^\bullet \scT_Q \rb
&:=\bigoplus_{i=0}^d H^i \lb B, \iota_\ast \bigwedge^i \scT_Q \rb,
\end{align}
where $T_i \subset H^i \lb B, \iota_\ast \bigwedge^i \scT_\bZ \rb$ is the torsion subgroup.
They have the graded ring structure induced by the wedge product.
Let further $X_\Sigmav$ be the complex toric manifold associated with $\Sigmav$, and $\iota \colon Y \hookrightarrow X_\Sigmav$ be an anti-canonical hypersurface.
For $0 \leq i \leq d$, we set
\begin{align}
H^{2i}_{\mathrm{amb}}\lb Y, \bZ \rb&:=\mathrm{Im} \lb \iota^\ast \colon H^{2i}(X_\Sigmav, \bZ) \to H^{2i}(Y, \bZ) \rb,\\
H^\bullet_{\mathrm{amb}}\lb Y, \bZ \rb&:=\bigoplus_{i=0}^d H^{2i}_{\mathrm{amb}}\lb Y, \bZ \rb.
\end{align}
The group $H^\bullet_{\mathrm{amb}}\lb Y, \bZ \rb$ also has the graded ring structure induced by the cup product.
Let $\Sigmav(1)$ denote the set of $1$-dimensional cones of $\Sigmav$, and $D_\rho$ be the restriction to $Y$ of the toric divisor on $X_\Sigmav$ corresponding to $\rho \in \Sigmav(1)$.
We write the primitive generator of $\rho \in \Sigmav(1)$ as $m_\rho \in \Delta \cap M$.
The following is the first main theorem of this paper.
This is a result of computation of the radiance obstruction $c_B \in H^1 \lb B, \iota_\ast \scT_\bR \rb$ of $B$.
See \pref{sc:integral} for the definition of radiance obstructions.

\begin{theorem}\label{th:1}
In the above setup, the following holds:
\begin{enumerate}
\item There is an injective graded ring homomorphism
\begin{align}
\psi \colon H^\bullet_{\mathrm{amb}}\lb Y, \bZ \rb \hookrightarrow  H^\bullet_\mathrm{f} \lb B, \iota_\ast \bigwedge^\bullet \scT_\bZ \rb.
\end{align}
\item The radiance obstruction $c_B \in H^1 \lb B, \iota_\ast \scT_\bR \rb$ of $B$ is given by
\begin{align}
	c_B= \sum_{\rho \in \Sigmav(1)} \check{h}(m_\rho)  \psi (D_\rho).
\end{align}
\end{enumerate}
\end{theorem}
The assumption that the fans $\Sigma, \Sigmav$ are unimodular is crucial in \pref{th:1}.1.
See \pref{rm:unimodular}.

We move on to the discussion on the relation with residual B-model variation of Hodge structure.
Let $K := \bC\lc t \rc$ be the convergent Laurent series field over $\bC$, i.e., the field of Laurent series $\sum_{j \in \bZ}c_j t^j$ that have only finitely many coefficients with negative index and whose positive part is convergent in a neighborhood of $t=0$. 
It has the standard non-archimedean valuation 
\begin{align}\label{eq:val}
	\mathrm{val} \colon K \longrightarrow \bZ \cup \{ \infty \},\quad k=\sum_{j \in \bZ}c_j t^j \mapsto \min \lc j \in \bZ \relmid c_j \ne 0 \rc.
\end{align}
We consider a Laurent polynomial $F=\sum_{m \in \Delta \cap M} k_m x^m \in K[x^\pm_{1}, \cdots, x^\pm_{d+1}]$ over $K$ such that the function $\Delta \cap M \to \bZ$ given by $m \mapsto \mathrm{val}(k_m)$ can be extended to a piecewise linear function $\check{g} \colon M_\bR \to \bR$ that is strictly convex on the fan $\Sigmav$.
We assume that the function
$\check{g}-\varphiv \colon M_\bR \to \bR$
is convex (not necessarily strictly convex) on $\Sigmav$ again.
In the following, let $B$ denote the $d$-sphere with an integral affine structure with singularities constructed from the data $\lb \Sigma, \check{g} \rb$.
\begin{comment}
This is also regarded as the space which is obtained by contracting the tropical hypersurface defined by the tropicalization $\mathrm{trop}(F)$ of $F$ given by
\begin{align}
\mathrm{trop}(F)(n):=\min_{m \in \Delta \cap M} \lc \mathrm{val}(k_m) + \la m, n \ra \rc.
\end{align}
\end{comment}
We write the image of the map 
\begin{align}\label{eq:psic1}
\psi \otimes_\bZ \bC \colon H^\bullet_\mathrm{amb} \lb Y, \bC \rb \hookrightarrow H^\bullet \lb B, \iota_\ast \bigwedge^\bullet \scT_\bC \rb
\end{align}
as $H^\bullet_{\psi} \lb B, \iota_\ast \bigwedge^\bullet \scT_\bC \rb$, and let $H^i_\mathrm{\psi} \lb B, \iota_\ast \bigwedge^i \scT_\bC \rb$ denote the image of $H^{2i}_\mathrm{amb} \lb Y, \bC \rb$ by the map \eqref{eq:psic1}.
We define $H^\mathrm{amb}_{A, \bZ, 0} \subset H^\bullet_\mathrm{amb}(Y, \bC)$ by
\begin{align}\label{eq:intext}
H^\mathrm{amb}_{A, \bZ, 0}:= \lc \lb 2 \pi \sqrt{-1} \rb^{-d} \widehat{\Gamma}_Y  
\cup \lb 2 \pi \sqrt{-1} \rb^{\frac{\mathrm{deg}}{2}} \mathrm{ch}(\iota^\ast \scE) \relmid \scE \in K(X_\Sigmav) \rc,
\end{align}
where $\widehat{\Gamma}_Y$ denotes the Gamma class of $Y$, and consider the lattice structure $H^\bullet_{\psi, \bZ} \lb B, \iota_\ast \bigwedge^\bullet \scT_\bC \rb$ of $H^\bullet_{\psi} \lb B, \iota_\ast \bigwedge^\bullet \scT_\bC \rb$ defined as the image of $H^\mathrm{amb}_{A, \bZ, 0}$ by the map \eqref{eq:psic1}.
We also consider the pairing
\begin{align}\label{eq:cup1}
Q \colon H^\bullet \lb B, \iota_\ast \bigwedge^\bullet \scT_\bC \rb \times H^\bullet \lb B, \iota_\ast \bigwedge^\bullet \scT_\bC \rb \to H^d \lb B, \iota_\ast \bigwedge^d \scT_\bC \rb \cong \bC
\end{align}
defined by
\begin{eqnarray}
Q(\alpha, \beta):=\left\{ \begin{array}{ll}
\lb 2 \pi \sqrt{-1} \rb^d (-1)^i \alpha \wedge \beta & i+j=d \\
0 & \mathrm{otherwise,}\\
\end{array} \right.
\end{eqnarray}
where $\alpha \in H^i \lb B, \iota_\ast \bigwedge^i \scT_\bC \rb, \beta \in H^j \lb B, \iota_\ast \bigwedge^j \scT_\bC \rb$.
We consider the following PLH.
See \pref{sc:lhs} for the definition of PLH.

\begin{definition-lemma}\label{dl:period}
The following triple $\lb H_\bZ^\mathrm{trop}, Q_\mathrm{trop}, \scrF_\trop \rb$ defines a polarized logarithmic Hodge structure on the standard log point $\lc 0 \rc$:
\begin{itemize}
\item the locally constant sheaf $H_\bZ^\mathrm{trop}$ on $\lc 0 \rc^\mathrm{log}$ whose stalk is isomorphic to $H^\bullet_{\psi, \bZ} \lb B, \iota_\ast \bigwedge^\bullet \scT_\bC \rb$ and the monodromy is given by the cup product of $\exp \lb - 2 \pi \sqrt{-1} c_B \rb$,
\item the $(-1)^d$-symmetric pairing
\begin{align}
Q_\mathrm{trop} \colon H^\bullet_{\psi, \bZ} \lb B, \iota_\ast \bigwedge^\bullet \scT_\bC \rb \times H^\bullet_{\psi, \bZ} \lb B, \iota_\ast \bigwedge^\bullet \scT_\bC \rb \to \bZ
\end{align}
defined as the restriction of the pairing \eqref{eq:cup1},
\item the decreasing filtration $\scrF_\trop=\lc \scrF^p_\mathrm{trop} \rc_{p=1}^d$ of $\scO_{\lc 0 \rc}^\mathrm{log} \otimes_\mathbb{Z} H_\bZ^\mathrm{trop} \cong \scO_{\lc 0 \rc}^\mathrm{log} \otimes_\mathbb{C} H^\bullet_{\psi} \lb B, \iota_\ast \bigwedge^\bullet \scT_\bC \rb$ given by 
\begin{align}\label{eq:filt}
\scrF^p_\mathrm{trop}:=\scO_{\lc 0 \rc}^\mathrm{log} \otimes_\bC \lb \bigoplus_{i=0}^{d-p} H^i_\mathrm{\psi} \lb B, \iota_\ast \bigwedge^i \scT_\bC \rb \rb.
\end{align}
\end{itemize}
We call this polarized logarithmic Hodge structure the tropical period of $B$.
\end{definition-lemma}

Let $\varepsilon$ be a positive real number that is smaller than the radius of convergence of every coefficient $k_m$ of the Laurent polynomial $F=\sum_{m \in \Delta \cap M} k_m x^m \in K[x^\pm_{1}, \cdots, x^\pm_{d+1}]$.
We set $D_\varepsilon :=\lc z \in \bC \relmid |z| < \varepsilon \rc$.
For each element $q \in D_\varepsilon \setminus \lc 0 \rc$, we consider the polynomial $f_q \in \bC[x^\pm_{1}, \cdots, x^\pm_{d+1}]$ obtained by substituting $q$ to $t$ in $F$.
Let $V_q$ be the complex hypersurface defined by $f_q$ in the complex toric manifold $X_\Sigma$ associated with $\Sigma$.
When $\varepsilon$ is sufficiently small, this is a smooth Calabi--Yau hypersurface.
The family of Calabi--Yau hypersurfaces $\lc V_q \rc_{q \in D_\varepsilon \setminus \lc 0 \rc}$ defines the residual B-model variation of Hodge structure on $D_\varepsilon \setminus \lc 0 \rc$.
It is extended to the LVPH on $D_\varepsilon$.
See \pref{sc:ext} for details about the extension.
The following is the second main theorem of this paper.
The claim that the tropical period of \pref{dl:period} is a PLH follows from this theorem.

\begin{theorem}\label{th:2}
The inverse image of the above LVPH by the inclusion $\lc 0 \rc \hookrightarrow D_{\varepsilon}$ is isomorphic to the tropical period of $B$.
\end{theorem}

There are several previous studies on the relationship between periods and tropical geometry.
It is known that the valuation of the $j$-invariant of an elliptic curve over a non-archimedean valuation field coincides with the cycle length of the tropical elliptic curve obtained by tropicalization \cite{MR2457725}, \cite{MR2570928}.
The definition of periods for general tropical curves was given in \cite{MR2457739}.
It was also shown in \cite{MR2576286} that the leading term of the period map of a degenerating family of Riemann surfaces is given by the period of the tropical curve obtained by tropicalization.
Ruddat--Siebert computed periods of toric degenerations constructed from wall structures \cite{RS19}.
They calculated the integrations of holomorphic volume forms over cycles constructed from tropical $1$-cycles on the intersection complex of the central fibers.
Abouzaid--Ganatra--Iritani--Sheridan \cite{AGIS18} also computed asymptotics of period integrals of toric Calabi--Yau hypersurfaces over cycles that are expected to be isotopic to Lagrangian cycles mirror to line bundles coming from the ambient space by using tropical geometry.

The organization of this paper is as follows:
In \pref{sc:integral}, we recall the definitions of integral affine manifolds with singularities and their radiance obstructions.
In \pref{sc:construction}, we review the construction in \cite[Section 3]{MR2198802} of integral affine spheres with singularities in the case of Calabi--Yau hypersurfaces.
In \pref{sc:proof1}, we give a proof of \pref{th:1}.
In \pref{sc:ku}, we recall the definition of polarized logarithmic Hodge structures on the standard log point, and review how a variation of polarized Hodge structure on a punctured disk is extended to a LVPH on the whole disk.
In \pref{sc:mirror}, we briefly recall the definitions of residual B-model/ambient A-model Hodge structure, and the mirror symmetry between them.
In \pref{sc:periods}, we discuss the relation between tropical periods and logarithmic Hodge theory.
\pref{th:2} is proved in this section.

\par
{\it Acknowledgment: } 
I am most grateful to my advisor Kazushi Ueda for his encouragement and helpful advice.
Some parts of this work were done during my visits to Yale University and University of Cambridge.
I thank Sam Payne for his encouragement on this project, valuable comments on an earlier draft of this paper, and the financial support for the visit.
I am also grateful to Mark Gross for helpful discussions and communications.
I learned the results of \cite{HZ02} and \cite{MR2187503}, and the base change theorem \cite[Theorem 4.2]{MR2669728}, \cite[Theorem 1.10]{FFR19} from him.
The visit to University of Cambridge was supported by JSPS Overseas Challenge Program for Young Researchers.
I also stayed at University of Geneva and attended the exchange student program ``Master Class in Geometry, Topology and Physics" by NCCR SwissMAP in the academic year 2016/2017.
During the stay, I learned a lot about tropical geometry from Grigory Mikhalkin.
I thank him for helpful discussions and sharing his insights and intuitions.
%I also thank NCCR SwissMAP for the financial support and the opportunity to study at University of Geneva.
I also thank Helge Ruddat for explaining the results of \cite{RS19} and \cite{Rud20}, and many beneficial comments on the draft of this paper.
\pref{rm:unimodular}, \pref{rm:res} and \pref{rm:rs19} are based on his comments.
I am also grateful to Yat-Hin Suen for helpful discussions about Gross--Siebert program.
I also thank the anonymous referee for pointing out some mistakes on an earlier draft of this paper and giving many comments which helped me to improve this paper.
This work was supported by IBS-R003-D1, Grant-in-Aid for JSPS Research Fellow (18J11281), and the Program for Leading Graduate Schools, MEXT, Japan. \\

%--------------------------------
\section{Integral affine structures with singularities}\label{sc:integral}
%--------------------------------

Let $N$ be a free $\bZ$-module of rank $d$, and set $N_\bR:=N \otimes_\bZ \bR$ and $\mathrm{Aff}(N_\bR):= N_\bR \rtimes \GL(N)$.
Note that the linear part of $\mathrm{Aff}(N_\bR)$ is integral, while the translational part is real.

\begin{definition}
An \emph{integral affine manifold} is a real topological manifold $B$ with an atlas of coordinate charts $\psi_i \colon U_i \to N_\bR$ such that the restriction of any transition function $\psi_i \circ \psi_j^{-1} \colon U_i \cap U_j \to N_\bR$ to any connected component of $U_i \cap U_j$ is contained in $\mathrm{Aff}(N_\bR)$.
\end{definition}

Let $B$ be an integral affine manifold.
We give an affine bundle structure to the tangent bundle $TB$ of $B$ as follows:
For each $U_i$ and $x \in U_i$, we set an affine isomorphism
\begin{align}
\theta_{i,x} \colon T_xB \to N_\bR,\quad v \mapsto \psi_i(x)+d\psi_i(x)v,
\end{align}
and define an affine trivialization by
\begin{align}
	\theta_i \colon TU_i \to U_i \times N_\bR,\quad (x, v) \mapsto (x, \theta_{i,x}(v)),
\end{align}
where $v \in T_xB$.
This gives an affine bundle structure to $TB$.
We write $TB$ with this affine bundle structure as $T^{\mathrm{aff}}B$.
Let $\scT_\bZ$ be the local system on $B$ of lattices of integral tangent vectors.
This is well-defined since the linear parts of all transition functions of $B$ are contained in $\GL(N)$.
We also set $\scT_\bR:=\scT_\bZ \otimes_\bZ \bR$.

\begin{definition} {\rm(\cite{{MR760977}})}
We choose a sufficiently fine open covering $\scU:=\lc U_i \rc_i$ of $B$ so that there is a flat section $s_i \in \Gamma (U_i, T^{\mathrm{aff}}B)$ for each $U_i$.
When we set $c_B((U_i, U_j)):=s_j-s_i$ for each $1$-simplex $(U_i, U_j)$ of $\scU$, the element $c_B$ becomes a \v{C}ech $1$-cocycle for $\scT_\bR$.
We call $c_B \in H^1(B, \scT_\bR)$ the \emph{radiance obstruction} of $B$.
\end{definition}

\begin{remark}
Radiance obstructions can also be described in terms of torsors as follows:
Let $\mathrm{Aff}_B$ be the sheaf of affine maps from $B$ to $\bR$, and $\scS$ be the sheaf of splittings of the projection $\mathrm{Aff}_B \to \mathrm{Aff}_B/\bR \cong \scT^\ast_\bR$, where $\scT^\ast_\bR:=\cHom\lb \scT_\bR, \bR \rb$.
The sheaf $\cHom \lb \mathrm{Aff}_B/\bR, \bR \rb=\scT_\bR$ acts on $\scS$ simply transitively, and $\scS$ becomes a $\scT_\bR$-torsor.
The radiance obstruction $c_B \in H^1(B, \scT_\bR)$ is the isomorphism class of $\scS$.
\end{remark}

\begin{definition} {\rm(\cite[Definition 1.15]{{MR2213573}})}
An \emph{integral affine manifold with singularities} is a topological manifold $B$ with an integral affine structure on $B_0:=B \setminus \Gamma$, where $\Gamma \subset B$ is a locally finite union of locally closed submanifolds of codimension greater than or equal to 2.
%We call $\Gamma$ the \emph{discriminant} of $B$.
\end{definition}

The following condition for integral affine manifolds with singularities was mentioned in \cite[Section 3.1]{MR2181810} as the fixed point property.

\begin{condition}\label{fix}
For any $x \in \Gamma$, there is a small neighborhood $U$ of $x$ such that the holonomy representation $\pi_1(U \setminus \Gamma) \to \mathrm{Aff}(N_\bR)$ has a fixed vector.
\end{condition}

In Condition \ref{fix}, note that $U \setminus \Gamma$ is connected, since the codimension of $\Gamma$ is greater than or equal to $2$.
Let $B$ be an integral affine manifold with singularities satisfying Condition \ref{fix}.
We write the complement of the singular locus as $\iota \colon B_0 \hookrightarrow B$.
Let further $\scT_\bZ$ be the local system on $B_0$ of lattices of integral tangent vectors.
We set $\scT_\bR:=\scT_\bZ \otimes_\bZ \bR$ again.

\begin{definition}
We choose a sufficiently fine covering $\lc U_i \rc_i$ of $B$ so that there is a flat section $s_i \in \Gamma (U_i \cap B_0, T^{\mathrm{aff}}B_0)$ for each $U_i$.
This is possible as long as we assume Condition \ref{fix}.
When we set $c_B((U_i, U_j)):=s_j-s_i$, the element $c_B$ becomes a \v{C}ech $1$-cocycle for $\iota_\ast \scT_\bR$.
We call $c_B \in H^1(B, \iota_\ast \scT_\bR)$ the \emph{radiance obstruction} of $B$.
\end{definition}

\begin{remark}
The inclusion $\iota \colon B_0 \hookrightarrow B$ induces a map $\iota^\ast \colon H^1(B, \iota_\ast \scT_\bR) \hookrightarrow H^1(B_0, \scT_\bR)$.
Then we can see $\iota^\ast \lb c_B \rb = c_{B_0}$ from the definitions.
\end{remark}

%--------------------------------
\section{Constructions of  integral affine spheres}\label{sc:construction}
%--------------------------------

In this section, we recall the construction in \cite[Section 3]{MR2198802} of integral affine spheres with singularities in the case of Calabi--Yau hypersurfaces.
We basically work on the same setup and use the same notations as in the introduction.
However, the fans $\Sigma, \Sigmav$ do not need to be unimodular in this section.

In general, for a piecewise linear convex function $g \colon M_\bR \to \bR$, its Newton polytope $\Deltav^g$ is defined by
\begin{align}
\Deltav^g:=\lc n \in N_\bR \relmid \la m,n\ra \geq -g(m) \rc.
\end{align}
We define $B^{\check{h}}:=\partial \lb \Deltav^{\check{h}} \rb = \partial \lb \Deltav + \Deltav^{\check{h}'} \rb$.
We also consider
\begin{align}\label{eq:tdelta}
\tilde{\Delta}:=\lc (m, l) \in M_\bR \oplus \bR \relmid m \in \Delta, l \geq \check{h}'(m) \rc,
\end{align}
and take a subdivision $\tilde{\Sigma}$ of the normal fan of $\tilde{\Delta}$ satisfying the following conditions:
\begin{enumerate}
\item the fan $\lc \sigma \cap (N_\bR \oplus \lc 0\rc) \relmid \sigma \in \tilde{\Sigma} \rc$ coincides with the fan $\Sigma$,
\item every $1$-dimensional cone of $\tilde{\Sigma}$ not contained in $N_\bR \oplus \lc 0\rc$ is generated by a primitive vector $(n, 1)$ with $n \in \Deltav^{\check{h}'} \cap N$.
\end{enumerate}
Such a subdivision $\tilde{\Sigma}$ is called \textit{good} in \cite[Definition 3.8]{MR2198802}.
The fan $\tilde{\Sigma}$ consists of the following three sorts of cones (\cite[Observation 3.9]{MR2198802}):
\begin{enumerate}
\item cones of the form $C(\mu) \times \lc 0 \rc$ where $\mu \subset \partial \Deltav$ and $C(\mu) \in \Sigma$,
\item cones of the form $C(\mu) \times \lc 0 \rc + C(\nu) \times \lc 1 \rc$ where $\mu \subset \partial \Deltav,  \nu \subset \partial \Deltav^{\check{h}'}$, $C(\mu) \in \Sigma$, and $\mu + \nu$ is contained in a face of $\Deltav+\Deltav^{\check{h}'}$, 
\item cones contained in $C(\Deltav^{\check{h}'} \times \lc 1\rc)$,
\end{enumerate}
where $C(\mu), C(\nu), C(\Deltav^{\check{h}'} \times \lc 1\rc)$ denote the cones over $\mu, \nu, \Deltav^{\check{h}'} \times \lc 1\rc$ respectively.
Cones of the second type is called \textit{relevant}.
We set
\begin{align}
\scR(\tilde{\Sigma})&:=\lc (\mu, \nu) \relmid  C(\mu) \times \lc 0 \rc + C(\nu) \times \lc 1 \rc \mathrm{\ is\ a\ relevant\ cone\ of\ } \tilde{\Sigma} \rc,\\
\scP(\tilde{\Sigma})&:=\lc \mu + \nu \relmid (\mu, \nu) \in \scR(\tilde{\Sigma}) \rc.
\end{align}
Then one has
\begin{align}
B^{\check{h}}=\bigcup_{(\mu, \nu) \in \scR(\tilde{\Sigma})} \mu + \nu,
\end{align}
and $\scP(\tilde{\Sigma})$ is a polyhedral subdivision of $B^{\check{h}}$ \cite[Proposition 3.12]{MR2198802}.

Take the barycentric subdivision $\mathrm{Bar}( \scP(\tilde{\Sigma}) )$ of $\scP(\tilde{\Sigma})$.
We consider the union $\Gamma(\tilde{\Sigma}) \subset B^{\check{h}}$ of all simplices of $\mathrm{Bar}( \scP(\tilde{\Sigma}) )$ that do not contain a vertex of $\scP(\tilde{\Sigma})$ and do not intersect the interior of a maximal face of $\scP(\tilde{\Sigma})$.
For each vertex $v \in \scP(\tilde{\Sigma})$, let $W_v$ be the union of interiors of all simplices of $\mathrm{Bar}( \scP(\tilde{\Sigma}) )$ containing $v$.
The vertex $v$ is written as $v=\mu + \nu$, where $(\mu, \nu) \in \scR(\tilde{\Sigma})$.
Note that the dimensions of $\mu, \nu$ are $0$.
Let $n_\mu \in N_\bR$ be the element such that $\lc n_\mu \rc= \mu$.
We define a chart
\begin{align}\label{eq:proj}
\psi_v \colon W_v \to N_\bR / (\bR \cdot n_\mu)
\end{align}
via the projection.
For each maximal-dimensional face $\tau \in \scP(\tilde{\Sigma})$ of $B$, the affine subspace $A_\tau$ of $N_\bR$ spanned by $\tau$ is written as
\begin{align}
A_\tau = \lc n \in N_\bR \relmid \check{h}'(m_\tau)+\la m_\tau, n \ra=0 \rc,
\end{align}
where $m_\tau$ is some element of  $\Delta \cap M$.
We define a chart
\begin{align}
\psi_ \tau \colon \mathrm{Int}(\tau) \hookrightarrow A_\tau
\end{align}
via the inclusion.
These charts $\psi_v$ and $\psi_\tau$ define an integral affine structure on $B^{\check{h}} \setminus \Gamma(\tilde{\Sigma})$ \cite[Proposition 3.14]{MR2198802}.

The monodromy of this integral affine structure is described as follows:
Let $v_0=\mu_0+\nu_0, v_1=\mu_1+\nu_1$ be two vertices of $\scP(\tilde{\Sigma})$, and $\tau_0, \tau_1$ be two maximal-dimensional faces of $\scP(\tilde{\Sigma})$ containing $v_0$ and $v_1$.
Let $\gamma$ be a loop that starts at $v_0$, passes through $\mathrm{Int}(\tau_0)$, $v_1$, $\mathrm{Int}(\tau_1)$ in this order, and comes back to the original point $v_0$.
Let further $\scT_\bZ$ be the sheaf on $B^{\check{h}} \setminus \Gamma(\tilde{\Sigma})$ of integral tangent vectors, and $\scT_{\bZ, v_0}$ be its stalk at $v_0$.

\begin{proposition} {\rm(\cite[Proposition 3.15]{{MR2198802}})}\label{pr:monodromy}
The parallel transport $T_\gamma \colon \scT_{\bZ, v_0} \to \scT_{\bZ, v_0}$ along the loop $\gamma$ is given by
\begin{align}
T_\gamma(n)=n+\la m_{\tau_1}-m_{\tau_0}, n \ra (n_{\mu_1}-n_{\mu_0}),
\end{align}
where we identify $\scT_{\bZ, v_0}$ with $N / (\bZ \cdot n_{\mu_0})$.
\end{proposition}
Let $\iota \colon B^{\check{h}} \setminus \Gamma(\tilde{\Sigma}) \hookrightarrow B^{\check{h}}$ denote the inclusion.
\begin{corollary}\label{cr:sct}
In the above situation, we have the following:
\begin{enumerate}
\item The integral affine manifold with singularities $(B^{\check{h}}, \scP(\tilde{\Sigma}))$ satisfies Condition \ref{fix}.
\item One has $\iota_\ast \bigwedge^d \scT_\bZ \cong \bZ$ and $\iota_\ast \bigwedge^i \scT_\bZ \cong \iota_\ast \bigwedge^{d-i} \scT^\ast_\bZ$.
\end{enumerate}
\end{corollary}
\begin{proof}
The space $B^{\check{h}}$ is homeomorphic to the $d$-dimensional sphere (e.g. \cite[Theorem 2.5]{MR2187503}).
The statements are obvious from this fact and  \pref{pr:monodromy}.
There is the same statement as the second one also in \cite[Theorem 3.23 (2)]{MR2669728}.
\end{proof}

\begin{theorem}{\rm(\cite[Theorem 3.16]{{MR2198802}})}
Assume that for any cones $C \in \Sigma$, $\check{C} \in \Sigmav$, the intersections $C \cap \partial \Deltav$ and $\check{C} \cap \partial \Delta$ are elementary simplices.
Then $(B^{\check{h}}, \scP(\tilde{\Sigma}))$ is simple in the sense of \cite[Definition 1.60]{MR2213573}.
\end{theorem}
In our setup (the case of hypersurfaces), the above theorem can be checked as follows:
For a cell $\tau \in \scP(\tilde{\Sigma})$ with $1 \leq \dim \tau \leq d-1$, we set
\begin{align}
\Omega(\tau)&:=\lc e \in \scP(\tilde{\Sigma}) \relmid \dim e=1, e \prec \tau \rc, \\
R(\tau)&:=\lc f \in \scP(\tilde{\Sigma}) \relmid \dim f=d-1, \tau \prec f \rc.
\end{align}
Fix a maximal-dimensional cell $\tau_0 \in \scP(\tilde{\Sigma})$ containing $\tau$.
For $e \in \Omega(\tau)$, the polytope $\Deltav_e(\tau)$ of \cite[Definition 1.58]{MR2213573} becomes the convex hull of
\begin{align}
\lc m_{\tau'} - m_{\tau_0} \in M_\bR \relmid \tau' \succ \tau, \dim \tau' =d \rc.
\end{align}
Therefore, the polytope $\Deltav_e(\tau)$ is independent of $e \in \Omega(\tau)$.
The cone $\check{C}$ generated by 
\begin{align}\label{eq:generators}
\lc m_{\tau'} \relmid \tau' \succ \tau, \dim \tau' =d \rc, 
\end{align}
is contained in $\Sigmav$.  
The intersection $\check{C} \cap \partial \Delta$ is the convex hull of \eqref{eq:generators}.
The translation of this by $- m_{\tau_0}$ coincides with $\Deltav_e(\tau)$.
From the assumption, we can see that $\Deltav_e(\tau)$ is an elementary polytope.

Similarly, fix a vertex $v_0 =\mu_0+\nu_0 \in \scP(\tilde{\Sigma})$ contained in $\tau$.
For $f \in R(\tau)$, the polytope $\Delta_f(\tau)$ of \cite[Definition 1.58]{MR2213573} becomes the convex hull of
\begin{align}
\lc n_{\mu'}-n_{\mu_0} \in N_\bR \relmid \mu'+\nu'=v' \prec \tau, \dim v'=0 \rc.
\end{align}
Therefore, the polytope $\Delta_f(\tau)$ is also independent of $f \in R(\tau)$.
The cone $C$ generated by 
\begin{align}\label{eq:generators2}
\lc n_{\mu'} \relmid \mu'+\nu'=v' \prec \tau, \dim v'=0  \rc, 
\end{align}
is contained in $\Sigma$.  
The intersection $C \cap \partial \Delta$ is the convex hull of \eqref{eq:generators2}.
The translation of this by $- n_{\mu_0}$ coincides with $\Delta_f(\tau)$.
From the assumption, we can see that $\Delta_f(\tau)$ is also an elementary polytope.
Hence, $(B^{\check{h}}, \scP(\tilde{\Sigma}))$ is simple.
The following claim is now obvious.

\begin{proposition}\label{pr:monopoly}
If the fan $\Sigmav \subset M_\bR$ (resp. $\Sigma \subset N_\bR$) is unimodular, then the polytope $\Deltav_e(\tau)$ (resp. $ \Delta_f(\tau)$) of \cite[Definition 1.58]{MR2213573} is a standard simplex for any $\tau \in \scP(\tilde{\Sigma})$ with $1 \leq \dim \tau \leq d-1$, and $e \in \Omega(\tau)$ (resp. $f \in R(\tau)$).
\end{proposition}

\begin{remark}
There is also another construction of integral affine spheres with singularities by Haase and Zharkov, which was discovered independently \cite{HZ02}, \cite{MR2187503}.
\end{remark}

%--------------------------------
\section{Proof of \pref{th:1}}\label{sc:proof1}
%--------------------------------

We work on the same setup and use the same notations as in the introduction.
Let $(B, \tilde{\scP}):=(B^{\check{h}}, \scP(\tilde{\Sigma}))$ be the integral affine sphere with singularities that we constructed in \pref{sc:construction}.
We take the barycentric subdivision of $\tilde{\scP}$, and let $U_\tau$ denote the open star of the barycenter of $\tau \in \tilde{\scP}$.
\begin{comment}
, i.e.,
\begin{align}
U_\tau:=\bigcup_{\sigma \in \math{Bar}(\scP')} \mathrm{Int}(\sigma),
\end{align}
where $\sigma$ is
\end{comment}
We consider the \v{C}ech cohomology group $\check{H}^k(\scU, \iota_\ast \bigwedge^k \scT_\bZ)$ with respect to the open covering $\scU:=\lc U_\tau \rc_{\tau \in {\tilde{\scP}}}$ of $B$.

We recall that the cohomology ring $H^\bullet \lb X_\Sigmav, \bZ \rb$ of the ambient toric variety $X_\Sigmav$ is described in terms of the fan $\Sigmav$.
Let $\Sigmav(1)$ denote the set of $1$-dimensional cones of $\Sigmav$.
We write the primitive generator of each cone $\rho \in \Sigmav(1)$ as $m_\rho \in \Delta \cap M$.
We also associate an indeterminate $x_\rho$ with each $\rho \in \Sigmav(1)$.
Consider the polynomial ring $\bZ \ld x_\rho: \rho \in \Sigmav(1)\rd$.
Let $I$ denote its ideal generated by the monomials $x_{\rho_1} \cdots x_{\rho_k}$ such that the convex hull of $\rho_1, \cdots,  \rho_k$ is not in $\Sigmav$.
Let further $J$ be the ideal generated by $\lc \Sigma_{\rho \in \Sigmav(1)} \la m_\rho, n \ra x_\rho \rc_{n \in N}$.
Then we have
\begin{align}\label{eq:tcoh}
H^\bullet \lb X_\Sigmav, \bZ \rb \cong \bZ \ld x_\rho: \rho \in \Sigmav(1)\rd / (I+J)
\end{align}
(cf.~e.g.~\cite[Theorem 10.8]{MR495499}).

For each $\rho \in \Sigmav(1)$, we write the subset of $B=\partial \Deltav^{\check{h}}$ given by
\begin{align}
\lc n \in N_\bR \relmid \check{h}(m_\rho)+\la m_\rho, n \ra=0 \leq \check{h}(m_{\rho'})+\la m_{\rho'}, n \ra\ \mathrm{for\ any\ } \rho' \in \Sigmav(1) \rc
\end{align}
as $\sigma(\rho) \in \scP$.
We consider the natural polyhedral structure $\scP$ of $B$ that consists of $\lc \sigma(\rho) \rc_{\rho \in \Sigmav(1)}$ and all their faces.
The correspondence $\rho \leftrightarrow \sigma \lb \rho \rb$ gives a bijection between $\Sigmav(1)$ and the set of maximal-dimensional faces in $\scP$.
Let $\scP(0)$ be the set of $0$-dimensional faces in $\scP$.
For each $v \in \scP(0)$ and $\rho \in \Sigmav(1)$, we define an element $n(v, \rho) \in N$ as follows:
Let $\lc \rho_i \rc_{i=1, \cdots, d+1}$ be the set of cones in $\Sigmav(1)$ such that $v=\bigcap_{i=1}^{d+1} \sigma(\rho_i)$.
We define $n(v, \rho) \in N$ by the following $d+1$ equations:
\begin{eqnarray}\label{eq:mv}
\la m_{\rho_i}, n(v, \rho) \ra:=\left\{ \begin{array}{ll}
-1 & \rho=\rho_i \\
0 & \mathrm{otherwise}\\
\end{array} \right.
\quad 1 \leq i \leq d+1.
\end{eqnarray}
For each element $\tau \in \tilde{\scP}$, let $\pi(\tau)$ denote the minimal face in $\scP$ containing $\tau$.

We construct a ring homomorphism
\begin{align}\label{eq:psi}
\psi: \bZ \ld x_\rho: \rho \in \Sigmav(1)\rd \to \bigoplus_{i=0}^d H^i \lb B, \iota_\ast \bigwedge^i \scT_\bZ \rb
\end{align}
as follows:
For each monomial $x_{\rho_1} \cdots x_{\rho_k}$, we take $k$ arbitrary maps 
\begin{align}\label{eq:xi}
\xi_i \colon \scP \to \scP(0),\quad 1 \leq i \leq k
\end{align}
such that $\xi_i(\sigma)$ is a vertex of $\sigma$ for any $\sigma \in \scP$, and define $\psi(x_{\rho_1} \cdots x_{\rho_k}) \in \check{H}^k(\scU, \iota_\ast \bigwedge^k \scT_\bZ)$ by
\begin{align}\label{eq:phik}
\psi(x_{\rho_1} \cdots x_{\rho_k})\lb \lb U_{\tau_0}, \cdots, U_{\tau_k} \rb \rb
:=\bigwedge_{i=1}^k \lc n \lb \xi_i \lb \pi(\tau_i) \rb, \rho_i \rb - n \lb \xi_i \lb \pi(\tau_{i-1}) \rb, \rho_i \rb \rc
\end{align}
for each $k$-simplex $\lb U_{\tau_0}, \cdots, U_{\tau_k} \rb$ of $\scU$ such that $\tau_0 \prec \cdots \prec \tau_k$.
Note that the intersection $U_{\tau_0} \cap \cdots \cap  U_{\tau_k}$ is non-empty if and only if we can have $\tau_0 \prec \cdots \prec \tau_k$ by reordering $\tau_0, \cdots , \tau_k$.
We also set $\psi(1):=1 \in \bZ \cong \check{H}^0(\scU, \bZ)$.
Then the ring homomorphism $\psi$ is uniquely determined.
We will check that this map $\psi$ gives the map of \pref{th:1} in the following.

\begin{example}
Suppose $d=1$, and take $e_1^\ast, e_2^\ast$ as a basis of the lattice $M$.
Let $\Delta \subset M_\bR$ be the convex hull of $(2e_1^\ast-e_2^\ast)$, $(-e_1^\ast+2e_2^\ast)$, and $(-e_1^\ast-e_2^\ast)$.
Let further $\Sigmav \subset M_\bR$ be the unimodular complete fan which is shown on the left side of \pref{fg:curve}.
We consider the piecewise linear function $\check{h} \colon M_\bR \to \bR$ that is strictly convex on $\Sigmav$, which is determined by
\begin{align}
\check{h}(m)=
\lc
\begin{array}{ll}
3 &\quad  m: \mathrm{a\ vertex\ of\ } \Delta, \\
2 &\quad  m \in \Delta \cap M \mathrm{\ that\ is\ not\ a\ vertex\ of\ } \Delta. \\
\end{array}
\right.
\end{align}
When we take the normal fan of $\tilde{\Delta}$ \eqref{eq:tdelta} as the fan $\tilde{\Sigma}$, the polyhedral structure $\tilde{\scP}$ coincides with the natural polyhedral structure $\scP$.
The space $\lb B, \tilde{\scP}=\scP \rb$ is shown on the right side of \pref{fg:curve}.
\begin{figure}[htbp]%[!ht]
\vspace{0.5cm}%\vspace{-14cm}
\hspace{1.8cm}
	\includegraphics[scale=0.7]{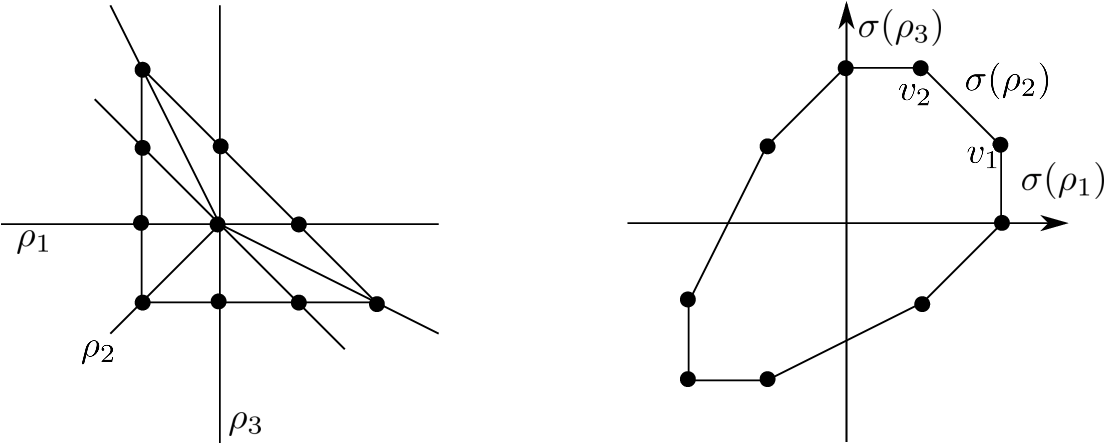}
\vspace{0.5cm}%\vspace{+15cm}
%\hspace{-2.2cm}
	\caption{The fan $\Sigmav$ and the space $\lb B, \scP \rb$}
	\label{fg:curve}
\end{figure}
Let $\rho_1, \rho_2, \rho_3 \in \Sigmav(1)$ be the cones of dimension $1$ generated by 
$m_{\rho_1}=-e_1^\ast, m_{\rho_2}=-e_1^\ast-e_2^\ast, m_{\rho_3}=-e_2^\ast$ respectively, and  $v_1, v_2 \in B$ be the vertices defined by $v_1:=\sigma(\rho_1) \cap \sigma(\rho_2)$ and $v_2:=\sigma(\rho_2) \cap \sigma(\rho_3)$.
One can check 
\begin{align}
n(v_1, \rho)=
\lc
\begin{array}{ll}
e_1-e_2 &\quad \rho=\rho_1 \\
e_2 &\quad \rho=\rho_2 \\
0 &\quad \mathrm{otherwise},
\end{array}
\right.
\quad
n(v_2, \rho)=
\lc
\begin{array}{ll}
e_1 &\quad \rho=\rho_2 \\
-e_1+e_2 &\quad \rho=\rho_3 \\
0 &\quad \mathrm{otherwise}.
\end{array}
\right.
\end{align}
We try to compute $\psi \lb x_{\rho_2} \rb$.
Take a map $\xi_1 \colon \scP \to \scP(0)$ so that 
\begin{itemize}
\item $\xi(\sigma)=\sigma$, when $\sigma \in \scP(0)$,
\item $\xi \lb \sigma(\rho_2)\rb=v_2$,
\item $\xi(\sigma)$ is a vertex of $\sigma$ that is not either $v_1$ or $v_2$, when $\sigma \nin \scP(0)$ and $\sigma \neq \sigma(\rho_2)$.
\end{itemize}
(Then the map $\xi_1$ satisfies \pref{cd:xi1} which we will assume later.)
For $\tau_0 \prec \tau_1 \in \scP$, we have
\begin{align}
\psi \lb x_{\rho_2} \rb \lb \lb U_{\tau_0}, U_{\tau_1} \rb \rb=
\lc
\begin{array}{ll}
n \lb v_2, \rho_2\rb -n \lb v_1, \rho_2\rb=e_1-e_2 &\quad \tau_0=v_1, \tau_1= \sigma(\rho_2)\\
0 &\quad \mathrm{otherwise}.
\end{array}
\right.
\end{align}
The vector $e_1-e_2$ is a primitive tangent vector on $U_{\sigma(\rho_2)}$, and it turns out that $\psi(x_{\rho_2})$ defines the class $1 \in \bZ \cong H^1(B, \iota_\ast \scT_\bZ)$.
\end{example}

\begin{lemma}\label{lm:1}
The map $\psi$ is a well-defined graded ring homomorphism, and independent of the choices of the maps $\xi_i$.
\end{lemma}
\begin{proof}
First, we check that the vector
\begin{align}\label{eq:vector}
\psi(x_{\rho})\lb (U_{\tau_0},U_{\tau_1}) \rb=n(\xi_1(\pi(\tau_1)), \rho)-n(\xi_1(\pi(\tau_0)), \rho)
\end{align}
is a section of $\iota_\ast \scT_\bZ$ over $U_{\tau_0} \cap U_{\tau_1}$ for any $\rho \in \Sigmav(1)$.
Since we have $\pi(\tau_0) \prec \pi(\tau_1)$,
when the face $\pi(\tau_1)$ is the intersection of facets $\lc \sigma(\rho_i) \rc_{i=1}^{l}$, the vertices $\xi_1(\pi(\tau_0)), \xi_1( \pi(\tau_1))$ are contained in $\bigcap_{i=1}^l \sigma(\rho_i)$.
Hence, we have
\begin{align}
\la m_{\rho_i}, n(\xi_1(\pi(\tau_0)), \rho) \ra=\la m_{\rho_i}, n(\xi_1(\pi(\tau_1)), \rho) \ra
\end{align}
for $1 \leq i \leq l$, and the vector \eqref{eq:vector} is contained in the plane defined by
\begin{align}
\la m_{\rho_i}, n \ra=0,\quad 1 \leq i \leq l.
\end{align}
This is the tangent space of $\pi(\tau_1)$.
On the other hand, from \pref{pr:monodromy}, we can see that the vector \eqref{eq:vector} is monodromy invariant with respect to any loop in $U_{\tau_0} \cap U_{\tau_1}$.
Therefore, the vector \eqref{eq:vector} is a section of $\iota_\ast \scT_\bZ$ over $U_{\tau_0} \cap U_{\tau_1}$.

We can also check that $\psi(x_\rho)$ is a cocycle.
For a given $2$-simplex $(U_{\tau_0}, U_{\tau_1}, U_{\tau_2})$, we have
\begin{align}\nonumber
\delta \lb \psi (x_\rho) \rb \lb (U_{\tau_0}, U_{\tau_1}, U_{\tau_2}) \rb &=\lc n(\xi_1(\pi(\tau_2)), \rho)-n(\xi_1(\pi(\tau_1)), \rho) \rc \\
& \quad -\lc n(\xi_1(\pi(\tau_2)), \rho)-n(\xi_1(\pi(\tau_0)), \rho) \rc \\ \nonumber
& \quad +\lc n(\xi_1(\pi(\tau_1)), \rho)-n(\xi_1(\pi(\tau_0)), \rho)\rc \\ 
&= 0.
\end{align}

Next, we check that $\psi(x_\rho)$ does not depend on the choice of the map $\xi_1$.
Take another map $\xi_1' \colon \scP \to \scP(0)$ such that $\xi'_1(\sigma)$ is a vertex of $\sigma$ for any $\sigma \in \scP$, and let $\psi'(x_\rho) \in  \check{H}^1(\scU, \iota_\ast \bigwedge^1 \scT_\bZ)$ be the cohomology class defined by the choice of $\xi'_1$.
We define $\phi(x_\rho) \in \check{C}^0(\scU, \iota_\ast \scT_\bZ)$ by
\begin{align}\label{eq:phi}
\phi(x_\rho)((U_{\tau})) :=  n(\xi'_1(\pi (\tau)), \rho)-n(\xi_1(\pi (\tau)), \rho)
\end{align}
for each $0$-simplex $(U_{\tau})$ of $\scU$.
We will show that the coboundary of $\phi(x_\rho)$ coincides with $\psi(x_\rho)-\psi'(x_\rho)$.

First, we check that $\phi(x_\rho)$ is certainly an element of $\check{C}^0(\scU, \iota_\ast \scT_\bZ)$.
When the face $\pi(\tau)$ is the intersection of facets $\lc \sigma(\rho_i) \rc_{i=1}^{l}$, the vertices $\xi_1(\pi(\tau)), \xi'_1( \pi(\tau))$ are contained in $\bigcap_{i=1}^l \sigma(\rho_i)$.
Hence, we have
\begin{align}
\la m_{\rho_i}, n(\xi_1(\pi(\tau)), \rho) \ra=\la m_{\rho_i} , n(\xi'_1(\pi(\tau)), \rho) \ra
\end{align}
for $1 \leq i \leq l$, and the vector \eqref{eq:phi} is contained in the plane defined by
\begin{align}
\la m_{\rho_i}, n \ra=0,\quad 1 \leq i \leq l.
\end{align}
This is the tangent space of $\pi(\tau)$.
On the other hand, from \pref{pr:monodromy} again, we can see that the vector \eqref{eq:phi} is monodromy invariant with respect to any loop in $U_{\tau}$.
Therefore, the vector \eqref{eq:phi} is a section of $\iota_\ast \scT_\bZ$ over $U_{\tau}$.

For any $1$-simplex $(U_{\tau_0}, U_{\tau_1})$ of $\scU$, one can get
\begin{align}\nonumber
\lb \psi'(x_\rho)-\psi(x_\rho) \rb ((U_{\tau_0}, U_{\tau_1}))
&= \lc n(\xi'_1(\pi(\tau_1)), \rho)-n(\xi'_1(\pi(\tau_0)), \rho) \rc \\ 
& \qquad - \lc n(\xi_1(\pi(\tau_1)), \rho)-n(\xi_1(\pi(\tau_0)), \rho) \rc \\ \nonumber
&= \lc n(\xi'_1(\pi(\tau_1)), \rho)-n(\xi_1(\pi(\tau_1)), \rho) \rc \\
& \qquad - \lc n(\xi'_1(\pi(\tau_0)), \rho)-n(\xi_1(\pi(\tau_0)), \rho) \rc \\
&= (\delta \phi(x_\rho))((U_{\tau_0}, U_{\tau_1})).
\end{align}
Hence, we have $\psi(x_\rho)= \psi'(x_\rho)$ in $\check{H}^1(\scU,\iota_\ast \scT_\bZ)$.

From the definition \eqref{eq:phik}, it is obvious that $\psi(x_{\rho_1} \cdots x_{\rho_k})$ coincides with the element $\bigwedge_{i=1}^k \psi(x_{\rho_i}) \in \check{H}^k(\scU, \bigwedge^k \iota_\ast \scT_\bZ)$.
Via the map $\bigwedge^k \iota_\ast \scT_\bZ \hookrightarrow \iota_\ast \bigwedge^k \scT_\bZ$, the element $\psi(x_{\rho_1} \cdots x_{\rho_k})$ defines an element of $\check{H}^k(\scU, \iota_\ast \bigwedge^k \scT_\bZ)$.
Since $\psi(x_{\rho_i})$ does not depend on the choice of the map $\xi_i: \scP \to \scP(0)$, the element $\psi(x_{\rho_1} \cdots x_{\rho_k})$ also does not depend on the choices of the maps $\xi_i: \scP \to \scP(0), 1 \leq i \leq k$.
\end{proof}

In the following, for a monomial $x_{\rho_1} \cdots x_{\rho_k}$, we choose $k$ maps $\xi_i \colon \scP \to \scP(0), 1 \leq i \leq k$, of \eqref{eq:xi} so that they satisfy the following condition:
\begin{condition}\label{cd:xi1}
For any $i \in \lc 1,\cdots, k\rc$ and any face $\sigma \in \scP$ such that $\sigma \nprec \sigma(\rho_i)$, one has $\xi_i(\sigma) \nin \sigma(\rho_i)$.
\end{condition}

When $\sigma \nprec \sigma(\rho_i)$, there is at least one vertex of $\sigma$ that is not in $\sigma(\rho_i)$.
We can choose such a vertex as $\xi_i(\sigma)$ for each $\sigma \in \scP$ such that $\sigma \nprec \sigma(\rho_i)$.
If we also choose an arbitrary vertex of $\sigma$ as $\xi_i(\sigma)$ for each $\sigma \in \scP$ such that $\sigma \prec \sigma(\rho_i)$, then the map $\xi_i$ satisfies \pref{cd:xi1}.
Hence, choosing maps $\xi_i$ so that they satisfy \pref{cd:xi1} is always possible.

\begin{lemma}\label{lm:2}
For a monomial $x_{\rho_1} \cdots x_{\rho_k}$, choose $k$ maps $\xi_i \colon \scP \to \scP(0), 1 \leq i \leq k$, of \eqref{eq:xi} so that they satisfy \pref{cd:xi1}.
Under such choices of $\xi_i$, if $\psi(x_{\rho_1} \cdots x_{\rho_k})\lb \lb U_{\tau_0}, \cdots, U_{\tau_k} \rb \rb \neq 0$, then one has
\begin{align}
\pi(\tau_i) \prec \bigcap_{j \geq i+1}^k \sigma(\rho_j)
\end{align}
for any $i \in \lc 0, \cdots, k-1 \rc$.
\end{lemma}
\begin{proof}
We prove this by induction on $i$.
We first show it for $i=k-1$, i.e., $\pi(\tau_{k-1}) \prec \sigma(\rho_k)$.
From the assumption $\psi(x_{\rho_1} \cdots x_{\rho_k})\lb \lb U_{\tau_0}, \cdots, U_{\tau_k} \rb \rb \neq 0$, we get
\begin{align}
n \lb \xi_k \lb \pi(\tau_k) \rb, \rho_k \rb - n \lb \xi_k \lb \pi(\tau_{k-1}) \rb, \rho_k \rb \neq 0.
\end{align}
If $\pi(\tau_{k}) \nprec \sigma(\rho_k)$, then we have $\xi_k \lb \pi(\tau_k) \rb \nin \sigma(\rho_k)$ from \pref{cd:xi1}.
We can see from \eqref{eq:mv} that $n \lb \xi_k \lb \pi(\tau_k) \rb, \rho_k \rb$ is equal to $0$.
Hence, we have $n \lb \xi_k \lb \pi(\tau_{k-1}) \rb, \rho_k \rb \neq 0$.
From \pref{cd:xi1} and \eqref{eq:mv} again, we get $\pi(\tau_{k-1}) \prec \sigma(\rho_k)$.
If $\pi(\tau_{k}) \prec \sigma(\rho_k)$, the relation $\pi(\tau_{k-1}) \prec \sigma(\rho_k)$ is obvious since $\pi(\tau_{k-1}) \prec \pi(\tau_{k})$.

Next, we show that for any $i_0 \in \lc 1, \cdots, k-1\rc$, if the statement holds for $i=i_0$, then it also holds for  $i=i_0-1$.
From the assumption $\psi(x_{\rho_1} \cdots x_{\rho_k})\lb \lb U_{\tau_0}, \cdots, U_{\tau_k} \rb \rb \neq 0$, we get
\begin{align}
n \lb \xi_{i_0} \lb \pi(\tau_{i_0}) \rb, \rho_{i_0} \rb - n \lb \xi_{i_0} \lb \pi(\tau_{i_0-1}) \rb, \rho_{i_0} \rb \neq 0.
\end{align}
If $\pi(\tau_{i_0}) \nprec \sigma(\rho_{i_0})$, then we have $\xi_{i_0} \lb \pi(\tau_{i_0}) \rb \nin \sigma(\rho_{i_0})$ from \pref{cd:xi1}.
We can see from \eqref{eq:mv} that $n \lb \xi_{i_0} \lb \pi(\tau_{i_0}) \rb, \rho_{i_0} \rb$ is equal to $0$.
Hence, we have $n \lb \xi_{i_0} \lb \pi(\tau_{i_0-1}) \rb, \rho_{i_0} \rb \neq 0$.
From \pref{cd:xi1} and \eqref{eq:mv} again, we get $\pi(\tau_{i_0-1}) \prec \sigma(\rho_{i_0})$.
On the other hand, by the induction hypothesis, we have
\begin{align}\label{eq:ind}
\pi(\tau_{i_0-1}) \prec \pi(\tau_{i_0}) \prec \bigcap_{j \geq i_0+1}^k \sigma(\rho_j).
\end{align}
Therefore, we get $\pi(\tau_{i_0-1}) \prec \bigcap_{j \geq i_0}^k \sigma(\rho_j)$.
If $\pi(\tau_{i_0}) \prec \sigma(\rho_{i_0})$, then we have $\pi(\tau_{i_0-1}) \prec \pi(\tau_{i_0}) \prec \sigma(\rho_{i_0})$.
By combining this with \eqref{eq:ind}, we get $\pi(\tau_{i_0-1}) \prec \bigcap_{j \geq i_0}^k \sigma(\rho_j)$ also in this case.
Hence, the statement holds also for $i=i_0-1$.
\end{proof}

\begin{lemma}\label{lm:3}
If $\psi(x_{\rho_1} \cdots x_{\rho_k}) \neq 0$, then we have $\bigcap_{j = 1}^k \sigma(\rho_j) \neq \emptyset$. 
\end{lemma}
\begin{proof}
Choose $k$ maps $\xi_i \colon \scP \to \scP(0), 1 \leq i \leq k$, of \eqref{eq:xi} so that they satisfy \pref{cd:xi1}.
When $\psi(x_{\rho_1} \cdots x_{\rho_k}) \neq 0$, there exists a $k$-simplex $\lb U_{\tau_0}, \cdots, U_{\tau_k} \rb$ of $\scU$ such that 
\begin{align}
\psi(x_{\rho_1} \cdots x_{\rho_k})\lb \lb U_{\tau_0}, \cdots, U_{\tau_k} \rb \rb \neq 0.
\end{align}
From \pref{lm:2} for $i=0$, i.e., $\pi(\tau_0) \prec \bigcap_{j =1}^k \sigma(\rho_j)$, we get $\bigcap_{j = 1}^k \sigma(\rho_j) \neq \emptyset$.
\end{proof}

\begin{lemma}\label{lm:4}
The kernel of the map $\psi$ contains $I+J$.
\end{lemma}
\begin{proof}
First, we show $I \subset \Ker(\psi)$.
For any cones $\rho_1, \cdots, \rho_k \in \Sigmav(1)$ such that the convex hull of $\rho_1, \cdots, \rho_k$ is not in $\Sigmav$, the set $\bigcap_{i=1}^{k} \sigma(\rho_i)$ is empty.
From \pref{lm:3}, we have $\psi \lb x_{\rho_1}, \cdots, x_{\rho_k} \rb = 0$.
Hence, we can see $I \subset \Ker(\psi)$.

Next, we show $J \subset \Ker(\psi)$.
For any $n \in N$, we have 
\begin{align}\label{eq:pdiv}
\psi \lb \sum_{\rho \in \Sigmav(1)} \la m_\rho, n \ra x_\rho \rb \lb \lb U_{\tau_0}, U_{\tau_1}\rb \rb=\sum_{\rho \in \Sigmav(1)} \la m_\rho, n \ra \lc n \lb \xi_1 \lb \pi(\tau_{1}) \rb, \rho \rb - n \lb \xi_1 \lb \pi(\tau_{0}) \rb, \rho \rb \rc.
\end{align}
On the other hand, for any vertex $v \in \scP(0)$, we have
\begin{align}\label{eq:div}
\sum_{\rho \in \Sigmav(1)} \la m_\rho, n \ra n \lb v, \rho \rb=-n.
\end{align}
This can be checked as follows:
Let $\lc \rho_i \rc_{i=1, \cdots, d+1}$ be the set of cones in $\Sigmav(1)$ such that $v=\bigcap_{i=1}^{d+1} \sigma(\rho_i)$.
Since the fan $\Sigmav$ is unimodular, the primitive generators $m_{\rho_1}, \cdots, m_{\rho_{d+1}}$ form a basis of the lattice $M$.
For any $i \in \lc 1, \cdots, d+1 \rc$, the pairings with $m_{\rho_i}$ of the both hand sides of \eqref{eq:div} are equal to $-\la m_{\rho_i},  n \ra$.
Therefore, the equation \eqref{eq:div} holds.
From \eqref{eq:div}, it turns out that \eqref{eq:pdiv} is $0$ for any $1$-simplex $\lb U_{\tau_0}, U_{\tau_1}\rb$ of $\scU$.
Therefore, we have $\psi \lb \sum_{\rho \in \Sigmav(1)} \la m_\rho, n \ra x_\rho \rb=0$ for any $n \in N$, and $J \subset \Ker(\psi)$.
\end{proof}

From \eqref{eq:tcoh} and \pref{lm:4}, we can see that the map $\psi$ \eqref{eq:psi} descends to the map from $H^\bullet \lb X_\Sigmav, \bZ \rb$.
This map will also be denoted by $\psi$,
\begin{align}\label{eq:psi2}
\psi \colon H^\bullet \lb X_\Sigmav, \bZ \rb \to \bigoplus_{i=0}^d H^i \lb B, \iota_\ast \bigwedge^i \scT_\bZ \rb.
\end{align}
%We consider the restriction $H^2 \lb X_\Sigmav, \bZ \rb \to H^1(B, \iota_\ast \scT_\bZ)$ of the map $\psi$.
The cohomology group $H^1(B, \iota_\ast \scT_\bZ)$ has the $d$-point function induced by the wedge product
\begin{align}
\bigwedge \colon H^1 \lb B, \iota_\ast \scT_\bZ \rb^{\otimes^d} \to H^d \lb B, \iota_\ast \bigwedge^d \scT_\bZ \rb \cong H^d \lb B, \bZ \rb \cong \bZ.
\end{align}
Here we use the isomorphism $\iota_\ast \bigwedge^d \scT_\bZ \cong \bZ$ of \pref{cr:sct}.
Choosing $\iota_\ast \bigwedge^d \scT_\bZ \cong \bZ$ amounts to choosing an orientation of $B$.
Furthermore, we need to choose an orientation of $B$ again in order to determine $H^d(B, \bZ) \cong \bZ$.
Here we choose the same orientation as we did for $\iota_\ast \bigwedge^d \scT_\bZ \cong \bZ$.
Then the isomorphism $H^d \lb B, \iota_\ast \bigwedge^d \scT_\bZ \rb \cong \bZ$ is independent of the orientation that we first choose for $\iota_\ast \bigwedge^d \scT_\bZ \cong \bZ$.
Let further $Y$ be an anti-canonical hypersurface of the complex toric variety $X_\Sigmav$.
For each $1$-dimensional cone $\rho \in \Sigmav(1)$, let $D_\rho$ denote the toric divisor on $X_\Sigmav$ corresponding to $\rho$.

\begin{lemma}\label{lm:5}
For any cones $\lc \rho_i \rc_{i=1}^d \subset \Sigmav(1)$, we have 
\begin{align}\label{eq:aim}
Y \cdot D_{\rho_1} \cdots D_{\rho_d} = \psi(x_{\rho_1}) \wedge \cdots \wedge \psi(x_{\rho_d}),
\end{align}
where the left hand side denotes the value in $H^{2(d+1)} \lb X_\Sigmav, \bZ \rb \cong \bZ$ of the cup product of the cohomology classes of $Y, D_{\rho_1}, \cdots, D_{\rho_d}$.
\end{lemma}
\begin{proof}
First, we show this for cones $\lc \rho_i \rc_{i=1}^d \subset \Sigmav(1)$ such that $\rho_i \neq \rho_j$ for any $i \neq j$ by computing the both sides of \eqref{eq:aim} explicitly.
We start with the right hand side.
From \pref{lm:4}, it turns out to be $0$ when the convex hull of $\lc \rho_i \rc_{i=1}^d$ is not in $\Sigmav$.
Therefore, we assume that the convex hull of $\lc \rho_i \rc_{i=1}^d$ is in $\Sigmav$ in the following.
In this case, the set $\bigcap_{j \geq 1}^d \sigma(\rho_j)$ is a $1$-dimensional face of $B$.
Let $v_0, v_1$ denote its vertices.
Let further $\sigma(\rho_0)$ be the facet of $B$ which contains $v_0$ and is different from $\sigma(\rho_i)\ 1 \leq i \leq d$, and $\sigma(\rho_{d+1})$ be the facet of $B$ which contains $v_{1}$ and is different from $\sigma(\rho_i)\ 1 \leq i \leq d$.
Since the fan $\Sigmav$ is unimodular, such facets $\sigma(\rho_0), \sigma(\rho_{d+1})$ uniquely exist.
We define $e_i \in N\ (0 \leq i \leq d)$ and $e_i' \in N\ (1 \leq i \leq d+1)$ by 
\begin{align}\label{eq:e1}
\la m_{\rho_j}, e_i \ra = \delta_{i, j}\ (0 \leq j \leq d), \quad  \la m_{\rho_j}, e_i' \ra = \delta_{i, j}\ (1 \leq j \leq d+1),
\end{align}
respectively.

For the monomial $x_{\rho_1} \cdots x_{\rho_d}$, we choose $d$ maps $\xi_i \colon \scP \to \scP(0), 1 \leq i \leq d$, of \eqref{eq:xi} so that they satisfy \pref{cd:xi1}, and in addition,

\begin{condition}\label{cd:xi2}
For any $i \in \lc 1, \cdots ,d \rc$ and any face $\sigma \in \scP$ such that $\sigma \prec \sigma(\rho_i)$, one has 
\begin{itemize}
\item $\xi_i(\sigma) = v_1$ if $v_1 \prec \sigma$,
\item $\xi_i(\sigma) = v_0$ if $v_1 \nprec \sigma$ and $v_0 \prec \sigma$.
 \end{itemize}
\end{condition}
Under such choices of maps $\xi_i$, we compute the sum of
\begin{align}\nonumber
\psi(D_{\rho_1}) \wedge \cdots \wedge \psi(D_{\rho_d})((U_{\tau_0}, \cdots, U_{\tau_d}))
&=\psi(D_{\rho_1})((U_{\tau_0}, U_{\tau_1})) \wedge \cdots \wedge \psi(D_{\rho_d})((U_{\tau_{d-1}}, U_{\tau_d}))\\ \label{eq:term}
&=\bigwedge_{i=1}^d \lc n \lb \xi_i \lb \pi(\tau_i) \rb, \rho_i \rb - n \lb \xi_i \lb \pi(\tau_{i-1}) \rb, \rho_i \rb \rc.
\end{align}
We care only about $d$-simplices $(U_{\tau_0}, \cdots, U_{\tau_d})$ such that $\psi(D_{\rho_1}) \wedge \cdots \wedge \psi(D_{\rho_d})((U_{\tau_0}, \cdots, U_{\tau_d})) \neq 0$.
By \pref{lm:2}, we know $\pi(\tau_0) \prec \bigcap_{j \geq 1}^d \sigma(\rho_j)$.
Hence, the face $\pi(\tau_0)$ contains either $v_0$ or $v_1$.
Since $\pi(\tau_0) \prec \pi(\tau_i)$, the face $\pi(\tau_i)$ also contains either $v_0$ or $v_1$ $(1 \leq i \leq d)$.
Therefore, we can see that $\pi(\tau_d)$ should be one of the facets $\sigma(\rho_i),\ 0 \leq i \leq d+1$.
Note that $\pi(\tau_d)$ is a facet of $B$, since the dimension of $\tau_d$ is $d$.

First, consider the case where $\pi(\tau_d)=\sigma(\rho_0)$ or $\pi(\tau_d)=\sigma(\rho_{d+1})$.
For any $i$, we have $\pi(\tau_i) \prec \pi(\tau_d)$, and $\pi(\tau_i) \prec \pi(\tau_d) \cap \lb \bigcap_{j \geq i+1}^d \sigma(\rho_j) \rb$ by \pref{lm:2}.
Since the dimension of $\pi(\tau_i)$ is greater than or equal to $i$, and the dimension of $\pi(\tau_d) \cap \lb \bigcap_{j \geq i+1}^d \sigma(\rho_j) \rb$ is $i$, we get 
\begin{align}\label{eq:face1}
	\pi(\tau_i) = \pi(\tau_d) \cap \lb \bigcap_{j \geq i+1}^d \sigma(\rho_j) \rb.
\end{align}
From this, we can see $\pi(\tau_i) \nprec \sigma(\rho_i)$.
Hence, we get $n(\xi_i \lb \pi(\tau_i) \rb, \rho_i) = 0$.
On the other hand, we have $\pi(\tau_{i-1}) \prec \sigma(\rho_i)$ by \pref{lm:2}.
Therefore, when $\pi(\tau_d)=\sigma(\rho_0)$, we have $\pi(\tau_{i-1})  \nsucc v_1$ and $\pi(\tau_{i-1}) \succ v_0$, and \eqref{eq:term} is equal to
\begin{align}\label{eq:term1}
\bigwedge_{i=1}^d - n \lb \xi_i \lb \pi(\tau_{i-1}) \rb, \rho_i \rb = \bigwedge_{i=1}^d -n(v_0, \rho_i)= \bigwedge_{i=1}^d e_i.
\end{align}
When $\pi(\tau_d)=\sigma(\rho_{d+1})$, we have $\pi(\tau_{i-1}) \succ v_1$, and \eqref{eq:term} is equal to
\begin{align}\label{eq:term2}
\bigwedge_{i=1}^d - n \lb \xi_i \lb \pi(\tau_{i-1}) \rb, \rho_i \rb = \bigwedge_{i=1}^d -n(v_1, \rho_i)= \bigwedge_{i=1}^d e_i',
\end{align}
where $e_i, e_i' \in N$ are the vectors defined in \eqref{eq:e1}.

Next, we consider the case where $\pi(\tau_d)=\sigma(\rho_{i_0})$ for some $i_0 \in \lc 1, \cdots, d \rc$.
We first show
\begin{align}\label{eq:face2}
\pi(\tau_i)=
\lc
\begin{array}{ll}
\sigma(\rho_{i_0}) \cap \lb \bigcap_{j \geq i+1}^d \sigma(\rho_j) \rb & i_0 \leq i \leq d, \\
\sigma(\rho_0) \cap \lb \bigcap_{j \geq i+1}^d \sigma(\rho_j) \rb & 0 \leq i \leq i_0-1. \\
\end{array}
\right.
\end{align}
For $i \geq i_0$, this can be shown by \pref{lm:2} and comparing the dimensions as we did in order to see \eqref{eq:face1}.
For $i \leq i_0-1$, we can check this as follows:
Since \eqref{eq:term} is not zero and $v_1 \prec \pi(\tau_{i_0})$, we have
\begin{align}
n(\xi_{i_0} \lb \pi(\tau_{i_0}) \rb, \rho_{i_0}) - n(\xi_{i_0} \lb \pi(\tau_{i_0-1}) \rb, \rho_{i_0}) = n(v_1, \rho_{i_0}) - n(\xi_{i_0} \lb  \pi(\tau_{i_0-1}) \rb, \rho_{i_0}) \neq 0.
\end{align}
From this and $\pi (\tau_{i_0-1}) \prec \pi (\tau_{i_0}) \prec \sigma(\rho_{i_0})$, it turns out that we have to have $\xi_{i_0} \lb  \pi(\tau_{i_0-1}) \rb=v_0$.
Hence, we get $v_1 \nprec \pi(\tau_{i_0-1})$.
This happens only when $\pi(\tau_{i_0-1}) \prec \sigma(\rho_0)$.
Therefore we have
\begin{align}
\pi (\tau_{i_0-1}) \prec \sigma(\rho_0) \cap \pi(\tau_{i_0}) \prec \sigma(\rho_0) \cap \lb \bigcap_{j \geq i_0}^d \sigma(\rho_j) \rb.
\end{align}
By comparing the dimensions of the faces, it turns out that \eqref{eq:face2} holds also for $i=i_0-1$.
For $i < i_0-1$, we know $\pi(\tau_i) \prec \pi(\tau_{i_0-1}) \prec \sigma(\rho_0)$.
Hence, from \pref{lm:2}, we have
\begin{align}
\pi(\tau_{i}) \prec \sigma(\rho_0) \cap \bigcap_{j \geq i+1}^d \sigma (\rho_j).
\end{align}
By comparing the dimensions of the faces again, it turns out that \eqref{eq:face2} holds also for $i < i_0-1$.

We now compute \eqref{eq:term} in the case where $\pi(\tau_d)=\sigma(\rho_{i_0})$ for some $i_0 \in \lc 1, \cdots, d \rc$.
By \eqref{eq:face2}, we know $\pi(\tau_i) \nprec \sigma (\rho_i)$ for $i \neq i_0$.
Hence, we have $n \lb \xi_i \lb \pi (\tau_i) \rb, \rho_i \rb=0$ for $i \neq i_0$.
Furthermore, from \eqref{eq:face2} again, we can also know $v_0 \prec \pi(\tau_i)$ and $v_1 \nprec \pi(\tau_i)$ for $i \leq i_0-1$, and $v_1 \prec \pi(\tau_i)$ for $i \geq i_0$.
Therefore, \eqref{eq:term} is equal to
\begin{align}\nonumber
\lb \bigwedge_{i=1}^{i_0-1} -n(v_0, \rho_i) \rb \wedge \lc n(v_1, \rho_{i_0}) - n(v_0, \rho_{i_0}) \rc \wedge \lb \bigwedge_{i=i_0+1}^d -n(v_1, \rho_i) \rb \\ \label{eq:term3}
=\lb \bigwedge_{i=1}^{i_0-1} e_i \rb \wedge \lb e_{i_0} - e_{i_0}' \rb \wedge \lb \bigwedge_{i=i_0+1}^{d} e_i' \rb,
\end{align}
where $e_i, e_i' \in N$ are the vectors defined in \eqref{eq:e1}.
From \eqref{eq:e1}, we can also see that each vector $e_i' \ (1 \leq i \leq d+1)$ can be written as
\begin{align}\label{eq:e2}
e_i'=e_i + s_i e_0 \ (1 \leq i \leq d),\quad e_{d+1}'=s_{d+1}e_{0}, 
\end{align}
where $s_i$ $(1 \leq i \leq d+1)$ are some integers.
Since these elements are primitive and $e_{d+1}' \neq e_{0}$, we have $s_{d+1}=-1$.
When we write the vertices $v_0, v_1$ as $v_0=\mu_0 + \nu_0$ and $v_1=\mu_1 + \nu_1$ , where $(\mu_0, \nu_0), (\mu_1, \nu_1) \in \scR(\tilde{\Sigma})$, we have 
\begin{align}
\mu_0&=\lc n \in N \relmid \la m_{\rho_i}, n \ra = - \varphiv(m_{\rho_i}) \mathrm{\ for\ } 0 \leq i \leq d \rc \\
&= \lc n \in N \relmid \la m_{\rho_i}, n \ra = -1 \mathrm{\ for\ } 0 \leq i \leq d \rc \\
&= \lc - \sum_{i=0}^{d} e_i =: n_{\mu_0} \rc.
\end{align}
Similarly we get $\mu_1=\lc - \sum_{i=1}^{d+1} e_i' =: n_{\mu_1} \rc$.
Hence, we have $\sum_{i=1}^{d+1} e_i=0, \sum_{i=1}^{d+1} e_i'=0$ on the charts of $U_{v_0}, U_{v_1}$ respectively.
Therefore, on $U_{v_0}$, \eqref{eq:term3} is equal to
\begin{align}
\lb \bigwedge_{i=1}^{i_0-1} e_i \rb \wedge \lb -s_{i_0} e_0 \rb \wedge \lb \bigwedge_{i=i_0+1}^{d} e_i+s_i e_0 \rb
&=-s_{i_0} \lb \bigwedge_{i=1}^{i_0-1} e_i \rb \wedge \lb \sum_{i=1}^d -e_i \rb \wedge \lb \bigwedge_{i=i_0+1}^{d} e_i \rb \\
&=s_{i_0} \bigwedge_{i=1}^{d} e_i. \label{eq:term4}
\end{align}

In either case, $\pi(\tau_d)=\sigma(\rho_0)$ or $\pi(\tau_d)=\sigma(\rho_{d+1})$ or $\pi(\tau_d)=\sigma(\rho_{i_0})$ $(i_0 \in \lc 1, \cdots, d \rc)$, there uniquely exists a sequence of faces $\tau_0 \prec \cdots \prec \tau_d$ satisfying \eqref{eq:face1} or \eqref{eq:face2}.
By thinking about the orientation of each simplex $(U_{\tau_0}, \cdots, U_{\tau_d})$, it turns out that both \eqref{eq:term1} and \eqref{eq:term2} are $1$, and \eqref{eq:term4} is $-s_{i_0}$ in $H^d \lb B, \iota_\ast \bigwedge^d \scT_\bZ \rb \cong \bZ$.
Hence, in total, we obtain
\begin{align}\label{eq:sum}
\psi(x_{\rho_1}) \wedge \cdots \wedge \psi(x_{\rho_d})=2-\sum_{i=1}^d s_i
\end{align}
in $H^d \lb B, \iota_\ast \bigwedge^d \scT_\bZ \rb \cong \bZ$.

\begin{comment}
Consider the monodromy of the vector $e_1$ with respect to a loop $\gamma$ that starts at $v_0$, passes through $\sigma(\rho_2)$, $v_1$, $\sigma(\rho_1)$ in this order, and comes back to $v_0$.
When the vector $e_1$ arrives at $v_1$, it becomes
\begin{align}
e_1'-s_1 e_0=e_1'+s_1 e_{d+1}'&=- \sum_{i=2}^{d+1} e_i' + s_1 e_{d+1}'.
\end{align}
When it comes back to $v_0$, this becomes
\begin{align}
-\sum_{i=2}^{d} \lb e_i +s_i e_0 \rb + (1-s_1) e_0
&=-\sum_{i=2}^{d} e_i + \lb 1-\sum_{i=1}^d s_i \rb e_0 \\
&=(e_0+e_1)+ \lb 1-\sum_{i=1}^d s_i \rb e_0 \\
&=e_1+ \lb 2-\sum_{i=1}^d s_i \rb e_0.
\end{align}
On the other hand, from \pref{pr:monodromy}, we know
\begin{align}
T_\gamma(e_1)=e_1+\la m_{\rho_1}- m_{\rho_2}, e_1 \ra \lb n_1 -n_0\rb=e_1+\lb n_{\mu_1} -n_{\mu_0} \rb.
\end{align}
Hence, \eqref{eq:sum} is equal to the integral length between $n_{\mu_0}$ and $n_{\mu_1}$.
\end{comment}

Next, we compute the left hand side of \eqref{eq:aim}.
A hypersurface in $X_\Sigmav$ defined by a polynomial whose Newton polytope is $\Deltav$ is an anti-canonical hypersurface.
When the convex hull of $\lc \rho_i \rc_{i=1}^d$ is not in $\Sigmav$, the intersection $Y \cap D_{\rho_1} \cap \cdots \cap D_{\rho_d}$ is obviously empty.
Hence, we have $Y \cdot D_{\rho_1} \cdots D_{\rho_d}=0$ and \eqref{eq:aim} holds in this case.
When the convex hull of $\lc \rho_i \rc_{i=1}^d$ is in $\Sigmav$ and $\rho_i \neq \rho_j$ for any $i \neq j$, the convex hull $\mathrm{Conv}(m_{\rho_1}, \cdots, m_{\rho_d})$ of the $d$ points $m_{\rho_1}, \cdots, m_{\rho_d}$ is a $(d-1)$-dimensional standard simplex on the boundary of $\Delta$.
When we restrict the polynomial defining $Y$ to $D_{\rho_i}$, all monomials except the monomials corresponding to $n \in \Deltav$ such that $\la m_{\rho_i}, n \ra=\min_{n' \in \Deltav} \la m_{\rho_i}, n' \ra=-1$ vanish.
When we restrict the polynomial defining $Y$ to $D_{\rho_1} \cap \cdots \cap D_{\rho_d}$, all monomials except the monomials corresponding to $n \in \Deltav$ such that $\la m_{\rho_i}, n \ra=-1$ for any $1 \leq i \leq d$ vanish.
Hence, the intersection $Y \cap D_{\rho_1} \cap \cdots \cap D_{\rho_d}$ is the zero locus in $D_{\rho_1} \cap \cdots \cap D_{\rho_d} \cong \bC P^1$ of a polynomial whose Newton polytope is the face $F$ of $\Deltav$ that is dual to the minimal face of $\Delta$ containing all $m_{\rho_1}, \cdots, m_{\rho_d}$.
The number of intersection points $Y \cap D_{\rho_1} \cap \cdots \cap D_{\rho_d}$ is the integral length of the face $F$.
This coincides with the affine length between $n_{\mu_0}$ and $n_{\mu_1}$.
On the other hand, we also have
\begin{align}\label{eq:nmu}
n_{\mu_1}-n_{\mu_0}= - \sum_{i=1}^{d+1} e_i' +\sum_{i=0}^{d} e_i =\lb 2 -\sum_{i=1}^d s_i \rb e_0.
\end{align}
From \eqref{eq:e1} and \eqref{eq:e2}, we also get
$m_{\rho_{d+1}}= -m_{\rho_0}+\sum_{i=1}^d s_i m_{\rho_i}$.
Since the point $m_{\rho_{d+1}}$ is in the polytope $\Delta$, we also know $\la m_{\rho_{d+1}}, n_{\mu_0}\ra \geq -1$.
The left hand side of this inequality equals
\begin{align}
\la -m_{\rho_0}+\sum_{i=1}^d s_i m_{\rho_i}, - \sum_{i=0}^d e_i \ra=1- \sum_{i=1}^d s_i.
\end{align}
Hence, we obtain $2 -\sum_{i=1}^d s_i \geq 0$.
From this and \eqref{eq:nmu}, it turns out that the affine length between $n_{\mu_0}$ and $n_{\mu_1}$ is $2 -\sum_{i=1}^d s_i$, which is equal to \eqref{eq:sum}.
Hence, \eqref{eq:aim} holds also in this case.

Lastly, we prove \eqref{eq:aim} for cones $\lc \rho_i \rc_{i=1}^d \subset \Sigmav(1)$ which do not necessarily satisfy $\rho_i \neq \rho_j$ for $i \neq j$.
We show it  by induction on the number of pairs $(\rho_i, \rho_j)$ such that $i < j$ and $\rho_i = \rho_j$.
Assume that \eqref{eq:aim} holds if the number of pairs $(\rho_i, \rho_j)$ such that $i < j$ and $\rho_i = \rho_j$ is less than or equal to $k \in \bN$.
We show that \eqref{eq:aim} holds also when it is $k+1$.
Let $\rho_{i_0} \in \Sigmav(1)$ be one of the cones such that there exists $i > i_0$ such that $\rho_i = \rho_{i_0}$.
There exists a primitive element $n \in N$ such that $\mathrm{div}(\chi^{n})= D_{\rho_{i_0}} + \sum_{\rho \notin \lc \rho_0, \cdots, \rho_{d} \rc} a_\rho D_\rho$ $(a_\rho \in \bZ)$.
Hence, we have
\begin{align}\label{eq:ind1}
Y \cdot D_{\rho_1} \cdots D_{\rho_{i_0}} \cdots D_{\rho_d} 
&= - \sum_{\rho \notin \lc \rho_0, \cdots, \rho_{d} \rc} a_\rho Y \cdot D_{\rho_1} \cdots D_{\rho} \cdots D_{\rho_d} \\ \label{eq:ind2}
&= - \sum_{\rho \notin \lc \rho_0, \cdots, \rho_{d} \rc} a_\rho \psi(x_{\rho_1}) \wedge \cdots \wedge \psi(x_{\rho}) \wedge \cdots \wedge \psi(x_{\rho_d}) \\
&= \psi(x_{\rho_1}) \wedge \cdots \wedge \psi \lb - \sum_{\rho \notin \lc \rho_0, \cdots, \rho_{d} \rc} a_\rho x_{\rho} \rb \wedge \cdots \wedge \psi(x_{\rho_d}) \\
&=\psi(x_{\rho_1}) \wedge \cdots \wedge \psi(x_{\rho_{i_0}}) \wedge \cdots \wedge \psi(x_{\rho_d}),
\end{align}
where we used the assumption of induction between \eqref{eq:ind1} and \eqref{eq:ind2}.
\end{proof}

Let $\mathrm{Ann}(\ld Y \rd) \subset H^\bullet(X_\Sigmav, \bZ)$ be the annihilator of the class $\ld Y \rd \in H^2(X_\Sigmav, \bZ)$.
We know
\begin{align}
H^\bullet_\mathrm{amb} \lb Y, \bZ \rb=H^\bullet(X_\Sigmav, \bZ) / \mathrm{Ann}(\ld Y \rd)
\end{align}
(cf.~e.g.~\cite[Proposition 8.1]{MR1265307}).
Let further $T \subset \bigoplus_{i=0}^d H^i \lb B, \iota_\ast \bigwedge^i \scT_\bZ \rb$ be the torsion subgroup.

\begin{lemma}
Concerning the map $\psi \colon H^\bullet \lb X_\Sigmav, \bZ \rb \to \bigoplus_{i=0}^d H^i \lb B, \iota_\ast \bigwedge^i \scT_\bZ \rb$ of \eqref{eq:psi2}, one has $\psi^{-1} \lb T \rb = \mathrm{Ann}(\ld Y \rd)$.
\end{lemma}
\begin{proof}
First, we show $\psi^{-1} \lb T \rb \subset \mathrm{Ann}(\ld Y \rd)$.
For any $x \in \psi^{-1} \lb T \cap H^i \lb B, \iota_\ast \bigwedge^i \scT_\bZ \rb \rb$ and $y \in H^{2(d-i)} \lb X_\Sigmav, \bZ \rb$, we have 
\begin{align}
\psi(x \cdot y) = \psi(x) \wedge \psi(y)=0.
\end{align}
The cohomology ring $H^\bullet \lb X_\Sigmav, \bZ \rb$ is generated by the toric divisors.
Hence, from \pref{lm:5}, we can see 
\begin{align}
\psi(x \cdot y)=Y \cdot x \cdot y
\end{align}
in $\bZ \cong H^d \lb B, \iota_\ast \bigwedge^d \scT_\bZ \rb \cong H^{2(d+1)} \lb X_\Sigmav, \bZ \rb$.
Therefore, we get $Y \cdot x \cdot y=0$ for any $y \in H^{2(d-i)} \lb X_\Sigmav, \bZ \rb$.
By the Poincar\'{e} duality for $X_\Sigmav$, we obtain $Y \cdot x=0$, i.e., $x \in \mathrm{Ann}(\ld Y \rd)$.
Hence, we get $\psi^{-1} \lb T \rb \subset \mathrm{Ann}(\ld Y \rd)$.

Next, we check $\psi^{-1} \lb T \rb \supset \mathrm{Ann}(\ld Y \rd)$.
We will show 
\begin{align}\label{eq:psiann}
\psi^{-1} \lb T \rb \cap H^{2i} \lb X_\Sigmav, \bZ \rb=\mathrm{Ann}(\ld Y \rd) \cap H^{2i} \lb X_\Sigmav, \bZ \rb
\end{align}
for any $0 \leq i \leq d$.
When $0 \leq i \leq d/2$, we can get $\mathrm{Ann}(\ld Y \rd) \cap H^{2i} \lb X_\Sigmav, \bZ \rb=0$ by the hard Lefschetz theorem.
Here, $\ld Y \rd$ is a K\"{a}hler class of $X_\Sigmav$, and defines the Lefschetz opeartor.
Since we know $\psi^{-1} \lb T \rb \subset \mathrm{Ann}(\ld Y \rd)$, we get \eqref{eq:psiann} for $0 \leq i \leq d/2$.
When $d/2 < i \leq d$, we know
\begin{align}\label{eq:yhigher}
\dim H^i \lb Y, \Omega^{i} \rb=\dim H^{2i}_\mathrm{amb} \lb Y, \bC \rb
\end{align}
by the hard Lefschetz theorem.
On the other hand, from \pref{cr:sct}, we have 
\begin{align}
\dim H^i \lb B, \iota_\ast \bigwedge^i \scT_\bC \rb = \dim H^i \lb B, \iota_\ast \bigwedge^{d-i} \scT^\ast_\bC \rb.
\end{align}
By \cite[Theorem 3.21]{MR2669728}, the dimension of $H^i \lb B, \iota_\ast \bigwedge^{d-i} \scT^\ast_\bC \rb$ is equal to the dimension of the logarithmic Dolbeault cohomology group of the log Calabi--Yau space associated with $B$.
Note that we can see from \pref{pr:monopoly} that the assumption of \cite[Theorem 3.21]{MR2669728} is satisfied.
By the base change theorem (see \cite[Theorem 4.2]{MR2669728}, with a gap fixed in \cite[Theorem 1.10]{FFR19}), we can see that this is equal to the dimension of the cohomology $H^i \lb X_t, \Omega^{d-i} \rb$ of a general fiber $X_t$ of the toric degeneration.
Since $X_t$ and $Y$ are defined by polynomials whose Newton polytopes are polar dual to each other, the dimension of $H^i \lb X_t, \Omega^{d-i} \rb$ is equal to \eqref{eq:yhigher} by \cite[Theorem 4.15]{MR1408560}.
Note that the stringy Hodge numbers of $X_t$ and $Y$ coincide with their ordinary Hodge numbers by \cite[Theorem 6.9]{MR1404917}, since $X_t$ and $Y$ are smooth.
Hence, we get $\dim H^{2i}_\mathrm{amb} \lb Y, \bC \rb=\dim H^i \lb B, \iota_\ast \bigwedge^i \scT_\bC \rb$.
If $\psi^{-1} \lb T \rb \cap H^{2i} \lb X_\Sigmav, \bZ \rb \subsetneq \mathrm{Ann}(\ld Y \rd) \cap H^{2i} \lb X_\Sigmav, \bZ \rb$, then we have
\begin{align}
\dim H^{2i}_\mathrm{amb} \lb Y, \bC \rb 
&= \dim \lb H^{2i}(X_\Sigmav, \bC) / \lc \lb \mathrm{Ann}(\ld Y \rd) \cap H^{2i} \lb X_\Sigmav, \bZ \rb \rb \otimes_\bZ \bC \rc \rb \\
&< \dim \lb H^{2i}(X_\Sigmav, \bC) / \lc \lb \psi^{-1} \lb T \rb \cap H^{2i} \lb X_\Sigmav, \bZ \rb \rb \otimes_\bZ \bC \rc \rb \\
&\leq \dim H^i \lb B, \iota_\ast \bigwedge^i \scT_\bC \rb,
\end{align}
which contradicts $\dim H^{2i}_\mathrm{amb} \lb Y, \bC \rb=\dim H^i \lb B, \iota_\ast \bigwedge^i \scT_\bC \rb$.
Therefore, we obtain \eqref{eq:psiann} also for $d/2 < i \leq d$.
\end{proof}

Therefore, the map $\psi$ \eqref{eq:psi2} defines the injective graded ring homomorphism
\begin{align}
\psi \colon H^\bullet_\mathrm{amb} \lb Y, \bZ \rb \hookrightarrow \lc \bigoplus_{i=0}^d \left. H^i \lb B, \iota_\ast \bigwedge^i \scT_\bZ \rb \rc \middle/ T \right. =:H^\bullet_\mathrm{f} \lb B, \iota_\ast \bigwedge^\bullet \scT_\bZ \rb
\end{align}
of \pref{th:1}.1.

\begin{remark}\label{rm:unimodular}
In general, there is a discrepancy between the dimension of $H^i \lb X_t, \Omega^{d-i} \rb$ and the dimension of $H^i \lb B, \iota_\ast \bigwedge^{d-i} \scT^\ast_\bC \rb$ (see \cite[Introduction]{MR2681794}), and the assumption that the fans $\Sigma, \Sigmav$ are unimodular is crucial in \pref{th:1}.1.
\end{remark}

\begin{proof}[Proof of \pref{th:1}.2]
We take a map of \eqref{eq:xi} $\xi_1 \colon \scP \to \scP(0)$ such that $\xi_1(\sigma) \prec \sigma$.
We take each chart $\psi_\tau \colon U_\tau \to N_\bR\ (\tau \in \tilde{\scP})$ so that we have $\psi_\tau (\xi_1(\pi(\tau)))=0$ when we enlarge the open set $U_\tau$ so that $U_\tau$ contains $\xi_1(\pi(\tau))$.
In order to specify the radiance obstruction of $B$, we choose the element in 
$\Gamma(U_\tau \cap B_0, T^{\mathrm{aff}}B_0) \cong \Gamma(U_\tau \cap B_0, \lb U_\tau \cap B_0 \rb \times N_\bR)$
given by $x \mapsto \lb x, \psi_\tau \lb x \rb \rb$ for each $U_\tau$.
Then the radiance obstruction $c_B$ is represented by the element of $C^1(\scU, \iota_\ast \scT)$ given by
\begin{align}
	c_B\lb (U_{\tau_0}, U_{\tau_1}) \rb = \xi_1(\pi(\tau_1))-\xi_1(\pi(\tau_0))
\end{align}
for each $1$-simplex $(U_{\tau_0},U_{\tau_1})$ of $\scU$.

Every vertex $v \in \scP(0)$ is determined as an element in $N_\bR$ by
\begin{align}
\la m_{\rho_i}, v \ra =-\check{h}(m_{\rho_i}),\quad 1 \leq i \leq d+1,
\end{align}
where $\rho_i\ (i=1, \cdots, d+1)$ are cones in $\Sigmav(1)$ such that $v=\bigcap_{i=1}^{d+1} \sigma(\rho_i)$.
Hence, one has
\begin{align}
v=\sum_{i=1}^{d+1} \check{h}(m_{\rho_i}) n (v, \rho_i) = \sum_{\rho \in \Sigmav(1)} \check{h}(m_\rho) n (v, \rho),
\end{align}
for any $v \in \scP(0)$.
Therefore, we get
\begin{align}
\lb \sum_{\rho \in \Sigmav(1)} \check{h}(m_\rho) \psi (D_\rho) \rb \lb (U_{\tau_0}, U_{\tau_1}) \rb&=
\sum_{\rho \in \Sigmav(1)} \check{h}(m_\rho) \lc n \lb \xi_1 \lb \pi(\tau_1) \rb, \rho \rb - n \lb \xi_1 \lb \pi(\tau_{0}) \rb, \rho \rb \rc \\
&=\xi_1(\pi(\tau_1))-\xi_1(\pi(\tau_0)) \\
&=c_B \lb (U_{\tau_0}, U_{\tau_1}) \rb.
\end{align}
\end{proof}

\begin{remark}
There is a work by Tsutsui \cite{Tsu20}, where he computes the radiance obstructions of tropical Kummer surfaces constructed by taking quotient of tropical tori.
\end{remark}

%--------------------------------
\section{Logarithmic Hodge theory}\label{sc:ku}
%--------------------------------

\subsection{PLH on the standard log point}\label{sc:lhs}

We recall the definition of polarized logarithmic Hodge structures (PLH) on the standard log point.
We refer the reader to \cite[Section 2.4]{MR2465224} for the definition of PLH on general fs logarithmic analytic spaces.

The standard log point is the point $\lc 0 \rc$ equipped with the logarithmic structure given by
\begin{align}\label{eq:str}
M_{\lc 0 \rc}:= \bC^\times \oplus \bN \to \scO_{\lc 0\rc}=\bC, \quad (a, n) \mapsto \left\{ \begin{array}{ll}
a & n = 0 \\
0 & n \neq 0. \\
\end{array} \right.
\end{align}
The Kato--Nakayama space $\lb\lc 0 \rc^\mathrm{log}, \scO_{\lc 0 \rc}^\mathrm{log} \rb$ in \cite{MR1700591} associated with the standard log point $\lc 0 \rc$ is the topological space $\lc 0 \rc^\mathrm{log} = S^1:=\lc z \in \bC \relmid |z|=1 \rc$ \cite[Section 1]{MR1700591} equipped with the sheaf of rings $\scO_{\lc 0 \rc}^\mathrm{log} = \bC \ld \log(q) \rd$ \cite[Section 3]{MR1700591}.
Here $ \bC \ld \log(q) \rd$ denotes the locally constant sheaf of $\bC$-algebras on $\lc 0 \rc^\mathrm{log} = S^1$ such that its stalk is the polynomial ring over $\bC$ in one variable and its generator denoted by $\log(q)$ has the monodromy action of $\pi_1 \lb S^1\rb \cong \bZ$ given by $\log(q) \mapsto \log(q)-2 \pi \sqrt{-1}$.
The logarithmic de Rham complex is given by
\begin{align}
0 \to \scO_{\lc 0 \rc}^\mathrm{log}=\bC \ld \log(q) \rd \xrightarrow{d} \omega_{\lc 0\rc}^{1, \mathrm{log}}=\bC \ld \log(q) \rd d \log(q) \to 0,
\end{align}
where $\bC \ld \log(q) \rd d \log(q):=\bC \ld \log(q) \rd \otimes_\bC \bC d \log(q)$, and $\bC d \log(q)$ is the constant sheaf on $\lc 0 \rc^\mathrm{log}$ of $1$-dimensional $\bC$-vector spaces whose generator is denoted by $d \log(q)$.

Let $w$ be an integer, and $\lc h^{p, q}\rc_{p, q \in \bZ}$ be a set of non-negative integers satisfying
\begin{itemize}
\item $h^{p,q}=0$ for almost\ all $p, q$,
\item $h^{p,q}=0$ if $p+q \neq w$,
\item $h^{p,q}=h^{q,p}$ for any $p, q$.
\end{itemize}
A {\em polarized logarithmic Hodge structure} \cite[Definition 2.4.7]{MR2465224} on the standard log point $\lc 0 \rc$ of weight $w$ and of Hodge type $\lb h^{p,q} \rb$ is a triple $\lb H_\bZ, Q, \scrF \rb$ consisting of 
\begin{itemize}
\item a locally constant sheaf of free $\bZ$-modules $H_\bZ$ of rank $\sum_{p,q}h^{p,q}$ on $\lc 0 \rc^\mathrm{log}$,
%\item an $\scO_{\lc 0\rc}= \bC$-vector space $\scM$ and a decreasing filtration $\lc \scM^p \rc_{p=\bZ}$ such that $\dim \scM^p /\scM^{p-1}=h^{p, p-w}$ for each $p$.
\item a non-degenerate $\bQ$-bilinear form $Q$ on $H_\bQ :=\bQ \otimes_\bZ H_\bZ$, which is symmetric for even $w$ and antisymmetric for odd $w$, and
\item a decreasing filtration $\scrF=\lc \scrF^p \rc_p$ of the $\scO_{\lc 0\rc}^\mathrm{log}$-module $\scO_{\lc 0\rc}^\mathrm{log} \otimes_\bZ H_\bZ$,
\end{itemize}
satisfying the following four conditions:
\begin{enumerate}
\item There exist a $\bC$-vector space $\scM$, a decreasing filtration $\lc \scM^p \rc_{p \in \bZ}$ on $\scM$, and an isomorphism
\begin{align}\label{eq:isom}
\scO_{\lc 0 \rc}^\mathrm{log} \otimes_\bC \scM  \cong \scO_{\lc 0 \rc}^\mathrm{log} \otimes_\bZ H_\bZ,
\end{align}
such that $\dim \scM^p /\scM^{p+1}=h^{p, w-p}$ for each $p$, and $\scO_{\lc 0 \rc}^\mathrm{log} \otimes_{\bC} \scM^p =\scrF^p$ under the identification \eqref{eq:isom}.
Here $\scM$ and $\scM^p$ are regarded as constant sheaves on $\lc 0\rc^\mathrm{log}$.
\item If $p+q>w$, then 
\begin{align}\label{eq:HR1}
Q(\scrF^p, \scrF^q)=0,
\end{align}
where $Q$ is regarded as the natural extension to an $\scO_{\lc 0 \rc}^\mathrm{log}$-bilinear form.
\item (Griffiths transversality)
\begin{align}
\lb d \otimes 1 \rb \lb \scrF^p \rb \subset \omega_{\lc 0 \rc}^{1, \mathrm{log}} \otimes_{\scO_{\lc 0 \rc}^\mathrm{log}} \scrF^{p-1}
\end{align}
for any $p$, where $d \otimes 1:=d \otimes 1_{H_\bZ} \colon \scO_{\lc 0 \rc}^\mathrm{log} \otimes_\bZ H_\bZ \to \omega_{\lc 0 \rc}^{1, \mathrm{log}} \otimes_\bZ H_\bZ$.
\item (Positivity) Let $y \in \lc 0 \rc^\mathrm{log}$ and $s \in \Hom_{\bC-\mathrm{alg}}\lb \scO_{\lc 0 \rc,y}^\mathrm{log}, \bC\rb$.
Let further $F(s)=\lc F^p(s) \rc_p$ be the decreasing filtration on the $\bC$-vector space $H_{\bC, y}:=\bC \otimes_\bZ H_{\bZ,y}$ defined by
\begin{align}
F^p(s):= \bC \otimes_{\scO_{\lc 0 \rc, y}^\mathrm{log}} \scrF^p_y, \quad \mathrm{with}\ s \colon \scO_{{\lc 0 \rc} ,y}^\mathrm{log} \to \bC.
\end{align}
Consider the map 
\begin{align}\label{eq:positive}
M_{\lc 0 \rc}=\bC^\times \oplus \bN \to \bC^\times, \quad (a, n) \mapsto \exp \lb s\lb \log(a) + n \log(q)\rb \rb,
\end{align}
where $\log(a)$ denotes the logarithm of $a \in \bC^\times$, which is determined up to $2 \pi \sqrt{-1} \bZ$, and $\log(a) + n \log(q) \in  \scO_{\lc 0 \rc,y}^\mathrm{log} / \lb 2 \pi \sqrt{-1} \bZ \rb$.
If the map \eqref{eq:positive} is sufficiently near to the structure morphism of the logarithmic structure \pref{eq:str} in the topology of simple convergence of $\bC$-valued functions, then $\lb H_{\bZ, y}, Q_y, F(s) \rb$ is a polarized Hodge structure of weight $w$ in the usual sense. \label{positivity}
\end{enumerate}

\subsection{Extension of variations of polarized Hodge structure}\label{sc:ext}

We briefly recall how variations of polarized Hodge structure on a punctured disk extend to logarithmic variations of polarized Hodge structure on the disk.
This subsection is based on \cite[Section 2.5.15]{MR2465224}.

Consider a variation of polarized Hodge structure $(H_\bZ, Q, \scrF)$ on the small punctured disk $D_\varepsilon \setminus \lc 0 \rc$, where $D_\varepsilon:= \lc q \in \bC \relmid |q| < \varepsilon \rc$.
It is a triple consisting of 
\begin{itemize}
\item a locally constant sheaf $H_\bZ$ on $D_\varepsilon \setminus \lc 0 \rc$ of free $\bZ$-modules of finite rank,
\item a bilinear pairing $Q \colon H_\bQ \times H_\bQ \to \bQ$ , where $H_\bQ:=H_\bZ \otimes_\bZ \bQ$, and
\item a decreasing filtration $\scrF=\lc \scrF^p\rc_p$ of $\scO_{D_\varepsilon \setminus \lc 0 \rc} \otimes_\bZ H_\bZ$ by $\scO_{D_\varepsilon \setminus \lc 0 \rc}$-submodules,
\end{itemize}
that satisfies the following conditions:
\begin{enumerate}
\item Each fiber $(H_{\bZ, x}, Q_x, \scrF_x)$ is a polarized Hodge structure.
\item Griffiths transversality: $\lb d \otimes_\bZ \id_{H_\bZ} \rb \lb \scrF^p\rb \subset \Omega_{D_\varepsilon \setminus \lc 0 \rc}^1 \otimes_{\scO_{D_\varepsilon \setminus \lc 0 \rc}} \scrF^{p-1}$.
\end{enumerate}
We fix a point $q_0 \in D_\varepsilon \setminus \lc 0 \rc$, and let $H_0$ denote the stalk of $H_\bZ$ at $q_0$.
We consider the map
\begin{align}
\bH_R \to D_\varepsilon,\quad z \mapsto \exp(2 \pi \sqrt{-1} z),
\end{align}
where $R$ is a positive real number such that $\exp(-2 \pi R) < \varepsilon$, and $\bH_R:= \lc z \in \bC \relmid \Im z > R\rc$.
We also fix a point $z_0 \in \bH_R$ such that $\exp(2 \pi \sqrt{-1} z_0)=q_0$.
We have an isomorphism between the pullback of $H_\bZ$ to $\bH_R$ and $H_0$ whose germ at $z_0$ is the identity map.
Here we regard $H_0$ as a constant sheaf on $\bH_R$.
Via this isomorphism, we get the associated period map
\begin{align}
\Phi \colon \bH_R \to D,
\end{align}
where $D$ denotes the Griffiths' period domain.

Assume that the monodromy $\gamma \colon H_0 \to H_0$ is unipotent.
We set $N:=\log(\gamma) \colon H_{0, \bQ} \to H_{0, \bQ}$, where $H_{0, \bQ}:=H_0 \otimes_\bZ \bQ$.
The map
\begin{align}\label{eq:descend}
\tilde{\Psi} \colon \bH_R \to \check{D},\quad z \mapsto \exp \lb -z N \rb \cdot \Phi(z)
\end{align}
descends to a holomorphic map $\Psi \colon D_\varepsilon \setminus \lc 0 \rc \to \check{D}$, where $\check{D}$ is the compact dual of $D$.
By Schmid's nilpotent orbit theorem \cite[Theorem 4.9]{MR0382272}, we know that this extends to a holomorphic map $\Psi \colon D_\varepsilon \to \check{D}$.

The logarithmic structure on $D_\varepsilon$ associated with the divisor $\lc 0 \rc \subset D_\varepsilon$ is given by
\begin{align}
\alpha \colon M_{D_\varepsilon} :=\bigcup_{n \geq 0} \scO_{D_\varepsilon}^\times \cdot q^n \hookrightarrow \scO_{D_\varepsilon}.
\end{align}
The Kato--Nakayama space $\lb D_\varepsilon^\mathrm{log}, \scO_{D_\varepsilon}^\mathrm{log} \rb$ associated with the logarithmic analytic space $\lb D_\varepsilon, M_{D_\varepsilon} \rb$ is
\begin{align}
D_\varepsilon^\mathrm{log} &= \lc (q, h) \relmid q \in D_\varepsilon, h \in \Hom(M_{D_\varepsilon, q}^\mathrm{gp}, S^1),  h(f)= \frac{f(q)}{|f(q)|} \mathrm{\ for\ any\ } f \in \scO_{D_\varepsilon}^\times \rc \\
&= \lb D_\varepsilon \setminus \lc 0 \rc \rb \sqcup S^1 \cong \ld 0, \varepsilon \rb \times S^1, \\
\scO_{D_\varepsilon}^\mathrm{log} &= \tau^{-1} \lb \scO_{D_\varepsilon} \rb \ld \log(q) \rd,
\end{align}
where $M_{D_\varepsilon, q}^\mathrm{gp}$ is the Grothendieck group associated with the stalk $M_{D_\varepsilon, q}$ at $q$, and $\tau \colon D_\varepsilon^\mathrm{log} \to D_\varepsilon$ is the natural map given by $(q, h) \mapsto q$.

The pushforward of the locally constant sheaf $H_\bZ$ on $D_\varepsilon \setminus \lc 0\rc$ by the inclusion $D_\varepsilon \setminus \lc 0\rc \hookrightarrow D_\varepsilon^\mathrm{log}$ is a locally constant sheaf on $D_\varepsilon^\mathrm{log}$ \cite[Theorem 3.1.2]{MR2046612}.
This will also be denoted by $H_\bZ$ in the following.
We have the isomorphism of $\scO_{D_\varepsilon}^\mathrm{log}$-modules
\begin{align}
\xi:=\exp \lb \lb 2 \pi \sqrt{-1} \rb^{-1} \log(q) \otimes N \rb \colon \scO_{D_\varepsilon}^\mathrm{log} \otimes_\bZ H_0 \xrightarrow{\sim} \scO_{D_\varepsilon}^\mathrm{log} \otimes_\bZ H_0, 
\end{align}
where $H_0$ is regarded as a constant sheaf on $D_\varepsilon^\mathrm{log}$.
The isomorphism
\begin{align}
H_{\bZ, q_0} = H_0 \to \left. \xi^{-1} \lb 1 \otimes H_0 \rb \right|_{q_0}, \quad v \mapsto \xi^{-1} (1 \otimes v)
\end{align}
of the stalks at $q_0$ preserves the actions of $\pi_1(D_\varepsilon^\mathrm{log}) \cong  \pi_1(D_\varepsilon \setminus \lc 0 \rc)$ (See the proof of \cite[Proposition 2.3.2]{MR2465224}).
Hence, it is extended uniquely to an isomorphism $H_\bZ \xrightarrow{\sim} \xi^{-1} \lb 1 \otimes H_0 \rb$ on $D_\varepsilon^\mathrm{log}$.
By taking the tensor product with $\scO_{D_\varepsilon}^\mathrm{log}$, we get the isomorphism of $\scO_{D_\varepsilon}^\mathrm{log}$-modules
\begin{align}\label{eq:diskisom}
\nu \colon \scO_{D_\varepsilon}^\mathrm{log} \otimes_\bZ H_\bZ \xrightarrow{\sim} \scO_{D_\varepsilon}^\mathrm{log} \otimes_\bZ \xi^{-1} \lb 1 \otimes H_0 \rb = \scO_{D_\varepsilon}^\mathrm{log} \otimes_\bZ H_0.
\end{align}

The holomorphic map $\Psi \colon D_\varepsilon \to \check{D}$ defines a filtration $\scrF_\Psi$ on $\scO_{D_\varepsilon} \otimes_\bZ H_0$.
We define $\tilde{\scrF}_\Psi:=\nu^{-1} \lb \scO_{D_\varepsilon}^\mathrm{log} \otimes_{\scO_{D_\varepsilon}} \scrF_\Psi \rb$.
The triple $\lb H_\bZ, Q, \tilde{\scrF}_\Psi \rb$ is a logarithmic variation of polarized Hodge structure (LVPH) on $D_\varepsilon$.
See \cite[Section 2.4.9]{MR2465224} for the definition of LVPH.
Its restriction to $D_\varepsilon \setminus \lc 0 \rc$ coincides with the original variation of polarized Hodge structure $(H_\bZ, Q, \scrF)$.
By taking the inverse image of the LVPH $\lb H_\bZ, Q, \tilde{\scrF}_\Psi \rb$ by the inclusion $\lc 0\rc \hookrightarrow D_\varepsilon$, we obtain a polarized Hodge structure on the standard log point, which we recalled in \pref{sc:lhs}.

\begin{remark}
Let $\Gamma$ be the subgroup of $\Aut(H_0, Q)$ generated the monodromy $\gamma \colon H_0 \to H_0$.
The map $\xi \circ \nu$ gives a $\Gamma$-level structure in the sense of \cite[Section 2.5.2]{MR2465224} (See Example 2 of \cite[Section 2.5.2]{MR2465224}.)
By \cite[Proposition 2.5.5]{MR2465224}, the conditions of Griffiths transversality and positivity for this PLH on the standard log point are equivalent to the following respectively:
\begin{enumerate}
\item (Griffiths transversality)
\begin{align}\label{eq:Gtrans}
N \cdot F_{\Psi, 0}^p \subset F_{\Psi, 0}^{p-1}
\end{align}
for any $p$, where $F_{\Psi, 0}=\lc F_{\Psi, 0}^p \rc_p$ is the fiber of the filtration $\scrF_{\Psi}$ at $0 \in D_\varepsilon$.
\item (Positivity)
When the imaginary part of $z \in \bC$ is sufficiently large, one has
\begin{align}\label{eq:positivity}
\exp \lb z N \rb \cdot  F_{\Psi, 0} \in D.
\end{align}
\end{enumerate}
\end{remark}

%--------------------------------
\section{Mirror symmetry for Calabi--Yau hypersurfaces}\label{sc:mirror}
%--------------------------------

In this section, we briefly recall the definitions of residual B-model/ambient A-model Hodge structure, and the mirror symmetry between them.
We refer the reader to \cite[Section 6]{MR3112512} for the details of the context of this chapter. 
There is also a review in the case of K3 hypersurfaces in \cite[Section 7]{Uedak3}.
For A-model Hodge structure, see also \cite[Section 8.5]{MR1677117}.

Let $M$ be a free $\bZ$-module of rank $d+1$ and $N:=\Hom(M, \bZ)$ be the dual lattice.
Let further $\Sigma \subset N_\bR, \Sigmav \subset M_\bR$ be unimodular fans whose fan polytopes $\check{\Delta} \subset N_\bR, \Delta \subset M_\bR$ are polar dual to each other.
We set $A:=(\Delta \cap M) \setminus \lc 0\rc$.
Let $\beta \colon \bZ^A \to M$ be the homomorphism sending the standard basis vectors $e_m$ $(m \in A)$ to $m \in M$.
We set $\bL:=\Ker(\beta)$.
The \emph{fan sequence} is the exact sequence
\begin{align}
0 \to \bL \to \bZ^A \xto{\beta} M \to 0
\end{align}
and the \emph{divisor sequence} is its dual
\begin{align}
0  \to N \xto{\beta^\ast} \lb \bZ^A \rb^\ast \to \bL^\ast \to 0.
\end{align}
One has
$\Pic \lb X_\Sigmav \rb \cong H^2(X_\Sigmav, \bZ) \cong \bL^\ast$.
We also set $\cM :=  \bL^\ast \otimes_\bZ \bCx, \bT:=N \otimes_\bZ \bCx$ and consider the exact sequence
\begin{align}
1 \to \bT \to (\bCx)^{A} \to \cM \to 1.
\end{align}

For a given element $\alpha=(a_m)_{m \in A} \in (\bCx)^{A}$, we associate the polynomial 
\begin{align}
	W_\alpha(x):=\sum_{m \in A} a_m x^m.
\end{align}
We consider the complex hypersurface
\begin{align}
Y_\alpha^\circ :=\lc x \in \bT \relmid W_\alpha(x)=1 \rc,
\end{align}
and its closure $Y_\alpha$ in the complex toric variety $X_{\Sigma} \supset \bT$ associated with the fan $\Sigma$.
Let $(\bCx)^{A}_{\mathrm{reg}}$ be the set of $\alpha \in (\bCx)^{A}$ such that $Y_\alpha$ is $\Sigma$-regular, i.e., the intersection of $Y_\alpha$ with any torus orbit of $X_{\Sigma}$ is a smooth subvariety of codimension one.
We consider the family of $\Sigma$-regular hypersurfaces given by the second projection 
\begin{align}
\tilde{\varphi} \colon \tilde{\frakY}:= \lc (x, \alpha) \in X_{\Sigma} \times (\bCx)^{A}_{\mathrm{reg}} \relmid W_\alpha(x)=1  \rc \to (\bCx)^{A}_{\mathrm{reg}},
\end{align}
and the action of $\bT$ to this family given by
\begin{align}
t \cdot (x, \alpha):=\lb t^{-1}x, (t^m a_m)_{m \in A} \rb,
\end{align}
where $t \in \bT$.
We write the quotient by this action as $\varphi \colon \frakY \to \scM_{\mathrm{reg}}$, where $\scM_{\mathrm{reg}}:=(\bCx)^{A}_{\mathrm{reg}} / \bT$.
The space $\scM_{\mathrm{reg}}$ is a Zariski open subset in $\cM$, and can be regarded as a parameter space of $\Sigma$-regular hypersurfaces whose Newton polytopes are $\Delta$.
Consider the residue part of $H^d(Y_\alpha, \bC)$ defined by
\begin{align}
H^d_\mathrm{res}(Y_\alpha, \bC) := \mathrm{Im} \lb \mathrm{Res} \colon H^0 \lb X_\Sigma, \Omega^{d+1}_{X_\Sigma} (\ast Y_\alpha) \rb \to H^d(Y_\alpha, \bC) \rb,
\end{align}
where $H^0 \lb X_\Sigma, \Omega^{d+1}_{X_\Sigma} (\ast Y_\alpha) \rb$ is the space of $(d+1)$-forms with arbitrary poles along $Y_\alpha$.
One can show \cite[Section 6.3]{MR3112512} that $H^d_\mathrm{res}(Y_\alpha, \bC)$ can be identified with the lowest weight component $W_d \lb H^d (Y_\alpha^\circ, \bC) \rb$ of the mixed Hodge structure on $H^d (Y_\alpha^\circ, \bC)$, and hence comes naturally with a $\bQ$-Hodge structure of weight $d$.
The {\em residual B-model VHS} \cite[Definition 6.5]{MR3112512} of the family $\check{\varphi} \colon \check{\frakY} \to \scM_{\mathrm{reg}}$ is the tuple $(\scrH_B, \nabla^B, H_{B,\bQ}, \scrF_B^\bullet, Q_B)$ consisting of
\begin{itemize}
\item the locally-free subsheaf $\scrH_B$ of $(R^d \varphiv_\ast \bC_\fYv) \otimes \cO_{\cM_\mathrm{reg}}$ whose fiber at $[\alpha]$ is $H^d_\mathrm{res}(Y_\alpha, \bC)$,
\item the Gauss--Manin connection $\nabla^B$ on $\scrH_B$,
\item the rational structure $H_{B, Q} \subset \Ker \nabla^B$ explained above,
\item the standard Hodge filtration $\scrF_{B,[\alpha]}^p = \bigoplus_{j \ge p} H_\mathrm{res}^{j,d-j}(Y_\alpha, \bC)$, and
\item the intersection form $Q_B \colon \scrH_B \otimes \scrH_B \to \scO_{\scM_\mathrm{reg}}$,
\begin{align}
(\omega_1, \omega_2) \mapsto (-1)^{d(d-1)/2} \int_{Y_\alpha} \omega_1 \cup \omega_2.
\end{align}
\end{itemize}
The composition
\begin{align}
H_{d+1} \lb \bT, Y_\alpha^\circ; \bC \rb
\xto{\ \partial \ } H_d(Y_\alpha^\circ, \bC)
\to H_d(Y_\alpha, \bC)
\xto{\ \mathrm{PD} \ } H^d(Y_\alpha, \bC)
\end{align}
gives a surjection
$\mathrm{VC} \colon H_{d+1} \lb \bT, Y_\alpha^\circ; \bC \rb \to H^d_\mathrm{res}(Y_\alpha, \bC)$ \cite[Lemma 6.6]{MR3112512}.
Here $\mathrm{PD}$ is the Poincar\'{e} duality isomorphism.
Note that the map $\mathrm{VC}$ is a surjection onto $H^d_\mathrm{res}(Y_\alpha, \bC)$, not onto $H^d(Y_\alpha, \bC)$.
The \emph{vanishing cycle integral structure} $H_{B, \bZ}^\mathrm{vc} \subset H_{B, \bQ}$ on the residual B-model VHS \cite[Definition 6.7]{MR3112512} is the image of $H_{d+1} \lb \bT, Y_\alpha^\circ; \bZ \rb$ by the map $\mathrm{VC}$.
The residual B-model VHS satisfies the Hodge--Riemann bilinear relations
\begin{align}\label{eq:HR}
Q_B(\scrF_B^p, \scrF_B^{d+1-p})=0, \quad \sqrt{-1}^{p-q} Q_B(\phi, \kappa(\phi))>0
\end{align}
for any $\phi \in \lc \scrF_B^p \cap \kappa (\scrF_B^{d-p}) \rc \setminus \lc 0\rc$ and forms a variation of polarized Hodge structure.
Here $\kappa$ denotes the real involution with respect to the real structure $H_{B, \bQ} \otimes_\bQ \bR$.

We move on to the A-model side.
Let $X_\Sigmav$ be the complex toric variety associated with $\Sigmav$, and $\iota \colon Y \hookrightarrow X_\Sigmav$ be an anti-canonical hypersurface.
We set
\begin{align}
U=\lc \tau \in H^2_\mathrm{amb} \lb Y, \bC \rb \relmid \Re \la \tau, d \ra \leq -M \ \mathrm{for\ any\ } d \in \mathrm{Eff}(Y) \setminus \lc 0 \rc \rc,
\end{align}
for some sufficiently large $M$.
Here $\mathrm{Eff}(Y)$ denotes the semigroup of effective curves on $Y$.
This open set $U$ is considered as a neighborhood of the large radius limit point.
We take a $\bC$-basis $\lc \eta_i \rc_{i=1}^l$ of $H^2_\mathrm{amb}(Y, \bC)$.
Let $\lc \tau_i \rc_{i=1}^l$ be the corresponding coordinates on $H^2_\mathrm{amb}(Y, \bC)$; $\tau =\sum_{i=1}^l \tau^i \eta_i$.
The {\em ambient A-model VHS} \cite[Definition 6.2]{MR3112512} of $Y$ is the tuple $(\scrH_A, \nabla^A, \scrF_A^\bullet, Q_A)$ consisting of
\begin{itemize}
\item the locally free sheaf $\scrH_A = H_\mathrm{amb}^\bullet(Y, \bC) \otimes \scO_U$,
\item the Dubrovin connection $\nabla^A = d + \sum_{i=1}^l (\eta_i \circ_\tau) \, d \tau^i \colon \scrH_A \to \scrH_A \otimes \Omega_{U}^1$,
\item the Hodge filtration $\scrF_A^p = H_\mathrm{amb}^{\le 2(d-p)}(Y, \bC) \otimes \scO_U$, and
\item the $(-1)^d$-symmetric pairing $Q_A : \scrH_A \otimes \scrH_A \to \scO_U$,
\begin{align}
(\alpha, \beta) \mapsto (2 \pi \sqrt{-1})^d \int_Y \lc (-1)^{\frac{\deg}{2}} \alpha \rc \cup \beta,
\end{align}
where $\deg$ denotes the degree as a class in $H_\mathrm{amb}^\bullet(Y, \bC)$.
\end{itemize}
Let $L_Y(\tau)$ be the quantum differential equation, i.e., the $\End(H^\bullet_\mathrm{amb}(Y, \bC))$-valued function on $U$ satisfying $\nabla^A_i L_Y(\tau)=0 \ (1 \le i \le l)$ and $L_Y(\tau)=\id + O(\tau)$.
The {\em ambient $\widehat{\Gamma}$-integral structure} \cite[Definition 6.3]{MR3112512} on the ambient A-model VHS is the local subsystem $H_{A, \bZ}^\mathrm{amb} \subset H_{A, \bC}:=\Ker \nabla^A$ defined by
\begin{align}\label{eq:latticea}
H^\mathrm{amb}_{A, \bZ}:=
\lc \lb 2 \pi \sqrt{-1} \rb^{-d} L_Y(\tau) \lb \widehat{\Gamma}_Y \cup \lb 2 \pi \sqrt{-1} \rb^{\frac{\mathrm{deg}}{2}} \mathrm{ch}(\iota^\ast \scE) \rb \relmid \scE \in K(X_\Sigmav) \rc,
\end{align}
where $\widehat{\Gamma}_Y$ denotes the Gamma class of $Y$.
The ambient A-model VHS also satisfies the Hodge--Riemann bilinear relations \eqref{eq:HR} and forms a variation of polarized Hodge structure.

We also take a basis $\lc p_i \rc_{i=1}^r$ of $\bL^\ast$ such that each $p_i$ is nef.
It determines coordinates $\lc q_i \rc_{i=1}^r$ on $\cM \supset \scM_{\mathrm{reg}}$.
Let $u_i \in H^2(X_\Sigmav, \bZ)$ be the Poincar\'{e} dual of the toric divisor $D_{\rho_i}$ corresponding to the one-dimensional cone $\rho_i \in \Sigmav$, and $v = \sum_{i=1}^m u_i$ be the anti-canonical class.
Givental's $I$-function is defined as the series
\begin{align}
I_{X_\Sigmav, Y}(q, z) = \exp \lb \frac{1}{z} \sum_{i=1}^r p_i  \log q_i \rb
\sum_{d \in \mathrm{Eff}(X_\Sigmav)} q^d \,
\frac{
	\prod_{k=-\infty}^{\la d, v \ra} (v + k z)
	\prod_{j=1}^m \prod_{k=-\infty}^0
	(u_j + k z)}
{\prod_{k=-\infty}^0
	(v + k z)
	\prod_{j=1}^m \prod_{k=-\infty}^{\la d, u_j \ra}
	(u_j + k z)},
\end{align}
which gives a multi-valued map from an open subset of $\scM \times \bCx$ to the classical cohomology ring $H^\bullet(X_\Sigmav,\bC)$.
If we write
\begin{align}\label{eq:expand}
	I_{X_\Sigmav, Y}(q, z) = F(q) + \frac{G(q)}{z} + O(z^{-2}),
\end{align}
the mirror map $\varsigma \colon \scM \to H^2_{\mathrm{amb}}(Y, \bC)$ is a multi-valued map given by
\begin{align}\label{eq:period}
	\iota^\ast \lb \frac{G(q)}{F(q)}\rb.
\end{align}

The group $H^2_{\mathrm{amb}}(Y, \bZ)$ acts on $H^2_{\mathrm{amb}}(Y, \bC)$ by the Galois action \cite[Proposition 2.3]{MR2553377}.
It preserves the connection $\nabla^A$, the pairing $Q_A$, and the integral structure $H_{A, \bZ}^\mathrm{amb}$, and defines the 
$H^2_{\mathrm{amb}}(Y, \bZ)$-action on the ambient A-model VHS $(\scrH_A, \nabla^A, H_{A, \bZ}^\mathrm{amb}, \scrF_A^\bullet, Q_A)$.
The ambient A-model VHS descends to the quotient $U/H^2_{\mathrm{amb}}(Y, \bZ)$.
The mirror map $\varsigma$ defines a single-valued map from a neighborhood of $q=0$ in $\scM$ to $H^2_{\mathrm{amb}}(Y, \bC)/ H^2_{\mathrm{amb}}(Y, \bZ)$.
The residual B-model VHS is isomorphic to the ambient A-model VHS via the mirror map $\varsigma$ including their integral structures in a neighborhood of $q=0$ \cite[Theorem 6.9]{MR3112512}.

%--------------------------------
\section{Tropical periods and logarithmic Hodge theory}\label{sc:periods}
%--------------------------------

We work on the same setup and use the same notations as in the introduction.
We consider a Laurent polynomial $F=\sum_{m \in \Delta \cap M} k_m x^m \in K[x^\pm_{1}, \cdots, x^\pm_{d+1}]$ over $K$ such that the map $\Delta \cap M \to \bZ$ given by $m \mapsto \mathrm{val}(k_m)$ can be extended to a piecewise linear function $\check{g} \colon M_\bR \to \bR$ that is strictly convex on $\Sigmav$.
Let $(B, \tilde{\scP}):=(B^{\check{g}}, \scP(\tilde{\Sigma}))$ be a $d$-sphere with an integral affine structure with singularities constructed in \pref{sc:construction}.
In the following, we assume $k_0=-1$ by multiplying an element of $K$ to $F$.

Since $\Sigmav$ is unimodular, when $\varepsilon \to 0$, the dominant terms of $f_q$ in the inverse image of a small open ball in $\bR^{d+1}$ by
\begin{align}
\mathrm{Log}_\varepsilon \colon \lb \bC^\times \rb^{d+1} \to \bR^{d+1}, \quad \lb x_1, \cdots, x_{d+1} \rb \mapsto \lb \log_{\varepsilon} | x_1 |, \cdots, \log_{\varepsilon} |x_{d+1} | \rb
\end{align}
are of the form $1+x_1'+\cdots + x_k'$ in a suitable coordinate system $\lb x_1', \cdots, x_{d+1}' \rb$.
Hence, when $\varepsilon$ is sufficiently small, the intersection of the hypersurface $V_q$ with the maximal-dimensional torus orbit $\lb \bC^\times \rb^{d+1}$ is smooth.
For the same reason, when $\varepsilon$ is sufficiently small, the intersection of $V_q$ with any lower-dimensional torus orbit is also smooth.
Therefore, the hypersurface $V_q$ is $\Sigma$-regular for any $q \in D_\varepsilon \setminus \lc 0 \rc$ when $\varepsilon$ is sufficiently small.
By replacing the real number $\varepsilon$ with a smaller one if necessary, we get a map $l$ given by
\begin{align}
l \colon D_\varepsilon \setminus \lc 0 \rc \to \scM_\mathrm{reg},\quad q \mapsto \ld \lb k_m (q) \rb_{m \in A} \rd,
\end{align}
where $A:=(\Delta \cap M) \setminus \lc 0\rc$ and $k_m(q)$ denotes the complex number obtained by substituting $q$ to $t$ in $k_m$.
By pulling back the residual B-model VHS over $\scM_\mathrm{reg}$ by the map $l$, we obtain a variation of polarized Hodge structure over $D_\varepsilon \setminus \lc 0 \rc$.
We write it also as $(\scrH_B, \nabla^B, H_{B,\bZ}^\mathrm{vc}, \scrF_B^\bullet, Q_B)$ in the following.

We fix a point $q_0 \in D_\varepsilon \setminus \lc 0 \rc$ and set 
\begin{align}
\check{\scF}&:=\lc \lc F^p \rc_{p=1}^d \relmid F^p \in \mathrm{Gr}(r_p, H_{B, \bC, q_0}^\mathrm{vc} ), F^1 \supset \cdots \supset F^d \rc, \\
\check{D}&:= \lc \lc F^p \rc_{p=1}^d \in \check{\scF} \relmid Q_B\lb F^p, F^{d-p+1} \rb=0 \rc,\\
D&:= \lc \lc F^p \rc_{p=1}^d \in \check{D} \relmid \forall v \in F^{p} \cap \overline{F}^q \setminus \lc 0 \rc, \lb \sqrt{-1}\rb^{p-q} Q_B \lb v, \overline{v} \rb >0 \rc,
\end{align}
where $H_{B, \bC, q_0}^\mathrm{vc}$ is the stalk of $H_{B, \bC}^\mathrm{vc}:=H_{B, \bZ}^\mathrm{vc} \otimes_\bZ \bC$ at $q_0$, and 
$r_p$ is the dimension of the stalk of $\scrF_{B}^p$.
We consider the map
\begin{align}
\pi \colon \bH_R \to D_\varepsilon,\quad z \mapsto \exp(2 \pi \sqrt{-1} z),
\end{align}
where $R$ is a positive real number such that $\exp(-2 \pi R) < \varepsilon$, and $\bH_R:= \lc z \in \bC \relmid \Im z > R\rc$.
We also fix a point $z_0 \in \bH_R$ such that $\exp(2 \pi \sqrt{-1} z_0)=q_0$.
We have an isomorphism between $\pi^\ast H_{B, \bZ}^\mathrm{vc}$ and the constant sheaf $H_{B, \bZ, q_0}^\mathrm{vc}$ on $\bH_R$ whose germ at $z_0$ is the identity map.
Via the isomorphism, we get the associated period map
\begin{align}
\Phi \colon \bH_R \to D.
\end{align}

\begin{proof}[Proof of \pref{th:2}]
Consider the pullback of the ambient A-model VHS over $U/H^2_{\mathrm{amb}}(Y, \bZ)$ that we discussed in the final paragraph of \pref{sc:mirror} by the map $\varsigma \circ l$.
We also write it as $(\scrH_A, \nabla^A, H_{A, \bZ}^\mathrm{amb}, \scrF_A^\bullet, Q_A)$.
We have the isomorphism of variation of polarized Hodge structure 
\begin{align}\label{eq:misom}
(\scrH_B, \nabla^B, H_{B,\bZ}^\mathrm{vc}, \scrF_B^\bullet, Q_B) \cong (\scrH_A, \nabla^A, H_{A, \bZ}^\mathrm{amb}, \scrF_A^\bullet, Q_A)
\end{align}
on $D_\varepsilon \setminus \lc 0\rc$ via the mirror isomorphism \cite[Theorem 6.9]{MR3112512}.
Here we replace the real number $\varepsilon$ with a smaller one again if necessary.
We will show the theorem by using this isomorphism.
In the following, $\lb H_\bZ^\mathrm{trop}, Q_\mathrm{trop}, \scrF_\mathrm{trop} \rb$ will denote the tropical period of $B$ which is defined in \pref{dl:period}.

First, we compute the monodromy of $H_{B,\bZ}^\mathrm{vc} \cong H_{A, \bZ}^\mathrm{amb}$.
The top term of the map $\varsigma \circ l \colon D_\varepsilon \setminus \lc 0 \rc \to H_\mathrm{amb}^2 (Y, \bC)$ is given by
\begin{align}\label{eq:top}
\sum_{i=1}^r p_i \log q_i (l (q) ),
\end{align}
where $\lc q_i \rc_{i=1}^r$ denote the coordinates on $\scM$ determined by the basis $\lc p_i \rc_{i=1}^r$ of $\bL^\ast$.
Let $\rho_m \in \Sigmav(1)$ be the cone whose primitive generator is $m \in A$.
Suppose $D_{\rho_m}=\sum_{i=1}^r b_{m,i} p_i$ in $H^2(X_\Sigmav, \bZ)$, where $b_{m,i} \in \bZ$.
Then we have
\begin{align}
q_i(l(q))&=q_i \lb \sum_{m \in A} D_{\rho_m} \otimes_\bZ k_m(q) \rb \\
&=q_i \lb \sum_{m \in A} \lb \sum_{j=1}^r b_{m, j} p_j \rb \otimes_\bZ k_m(q) \rb \\
&=q_i \lb \sum_{j=1}^r p_j \otimes_\bZ \prod_{m \in A} k_m(q)^{b_{m, j}} \rb \\
&=\prod_{m \in A} k_m(q)^{b_{m,i}}.
\end{align}
Hence, we obtain
\begin{align}
\sum_{i=1}^r p_i \log q_i \lb l(q)\rb 
&=\sum_{i=1}^r p_i \lb \sum_{m \in A} b_{m,i} \log \lb  k_m(q) \rb \rb\\
&\sim \sum_{i=1}^r p_i 
\lb
\sum_{m \in A}
b_{m,i} \log q \cdot \mathrm{val}(k_m) 
\rb\\
&=\log q \sum_{i=1}^r  p_i \lb \sum_{m \in A} b_{m,i} \mathrm{val}(k_m) \rb.
\end{align}
From this and \eqref{eq:top}, we can see that for each coordinate $\tau^i$, the map $\varsigma \circ l$ goes around the large radius limit $\sum_{m \in A} b_{m,i} \mathrm{val}(k_m)$ times.
The monodromy around the large radius limit with respect to each coordinate $\tau^i$ is given by the cup product of $\exp \lb -2 \pi \sqrt{-1} p_i \rb$ \cite[Theorem 10.2.4]{MR1677117}.
(See also \cite[Proposition 2.10 $\rm(\hspace{.08em}ii\hspace{.08em})$]{MR2553377}.)
Hence, it turns out that the monodromy of flat sections of $\scrH_{A}$ is given by the cup product of
\begin{align}\label{eq:monodromy}
\prod_{i=1}^r \exp \lb -2 \pi \sqrt{-1} p_i \lb \sum_{m \in A} b_{m,i} \mathrm{val}(k_m) \rb \rb &= \exp \lb -2 \pi \sqrt{-1} \sum_{i=1}^r p_i \lb \sum_{m \in A} b_{m,i} \mathrm{val}(k_m) \rb \rb \\
&=\exp \lb -2 \pi \sqrt{-1} \lb \sum_{m \in A} \mathrm{val}(k_m) D_{\rho_m} \rb \rb. \label{eq:monodromy}
\end{align}
From \pref{th:1}, we can see that the radiance obstruction $c_B$ of $B$ is given by
\begin{align}\label{eq:rd}
c_B=\sum_{m \in A} \mathrm{val}(k_m) \psi(D_{\rho_m}).
\end{align}
Hence, \eqref{eq:monodromy} is equal to
\begin{align}\label{eq:monodromy2}
\exp \lb -2 \pi \sqrt{-1} \psi^{-1} (c_B) \rb.
\end{align}
This monodromy is unipotent, and the variation of polarized Hodge structure \pref{eq:misom} on $D_\varepsilon \setminus \lc 0 \rc$ is extended to a logarithmic variation of polarized Hodge structure on $D_\varepsilon$ as we recalled in \pref{sc:ext}.

Next, we show $(H_{B, \bZ}^\mathrm{vc}, Q_B) \cong (H_\bZ^\mathrm{trop}, Q_\mathrm{trop})$.
We consider the canonical extension of Deligne in \cite{MR0417174}.
(One can also find its definition, for instance, in \cite[Definition 11.4]{MR2393625}.)
The canonical extension $\widetilde{\scrH}_A$ for $\scrH_A$ to $D_\varepsilon$ is constructed as follows:
We set
\begin{align}
N:=\log \lb \exp \lb - 2 \pi \sqrt{-1} \psi^{-1} (c_B) \rb \rb=- 2 \pi \sqrt{-1} \psi^{-1} (c_B),
\end{align}
and $H_A :=\Ker \nabla^A$.
For a section $s \in H^0 \lb \bH_R, \pi^\ast H_A \rb$, we define a holomorphic section $\varphi(s) \in H^0 \lb \bH_R, \pi^\ast \scrH_{A} \rb$ by
\begin{align}
\varphi(s)(z):=\exp(-z N) \cdot s(z).
\end{align}
This section $\varphi(s)$ is invariant under $z \mapsto z+1$, and descends to the section $\tilde{\varphi}(s) \in H^0 \lb D_\varepsilon \setminus \lc 0\rc, \scrH_{A} \rb$ defined by
\begin{align}\label{eq:dsection}
\tilde{\varphi}(s)(q):=\exp \lb - (2 \pi \sqrt{-1})^{-1} \log q \cdot N \rb \cdot s \lb (2 \pi \sqrt{-1})^{-1} \log q \rb.
\end{align}
The canonical extension $\widetilde{\scrH}_A$ is given by $\widetilde{\scrH_A}:= \tilde{\varphi} \lb H^0 \lb \bH_R, \pi^\ast H_A \rb \rb \otimes_\bC \scO_{D_\varepsilon}$.

For $\scE \in K(X_\Sigmav)$, we set
\begin{align}
\fraks(\scE)(\tau):= \lb 2 \pi \sqrt{-1} \rb^{-d} L_Y(\tau) \lb \widehat{\Gamma}_Y  
\cup \lb 2 \pi \sqrt{-1} \rb^{\frac{\mathrm{deg}}{2}} \mathrm{ch}(\iota^\ast \scE) \rb.
\end{align}
This is a flat section of $\scrH_A$ which is in the ambient $\widehat{\Gamma}$-integral structure $H^\mathrm{amb}_{A, \bZ}$ \pref{eq:latticea}.
The ambient $\widehat{\Gamma}$-integral structure $H_{A, \bZ}^\mathrm{amb}$ induces the integral structure $\widetilde{\scrH}_{A, \bZ}$ of $\widetilde{\scrH}_A$ defined by
\begin{align}
\widetilde{\scrH}_{A, \bZ}:= \lc \tilde{\varphi} \lb s(\scE) \rb \relmid \scE \in K(X_\Sigmav) \rc \subset \widetilde{\scrH}_A,
\end{align}
where $s(\scE):=(\varsigma \circ l \circ \pi)^\ast \lb \fraks \lb \scE \rb \rb$.
From \cite[Corollary 10.2.6]{MR1677117}, we can get
\begin{align}\label{eq:canon}
\tilde{\varphi}(s(\scE))(0) = \lb 2 \pi \sqrt{-1} \rb^{-d} \widehat{\Gamma}_Y  
\cup \lb 2 \pi \sqrt{-1} \rb^{\frac{\mathrm{deg}}{2}} \mathrm{ch}(\iota^\ast \scE).
\end{align}
Therefore, the restriction $\widetilde{\scrH}_{A, \bZ}(0)$ of $\widetilde{\scrH}_{A, \bZ}$ to $0 \in D_\varepsilon$ coincides with $H^\mathrm{amb}_{A, \bZ, 0}$ defined in \eqref{eq:intext}, and we have an identification
\begin{align}
\pi^\ast H_{A, \bZ}^\mathrm{amb} \cong \widetilde{\scrH}_{A, \bZ}(0) = H^\mathrm{amb}_{A, \bZ, 0},
\end{align}
given by $s(\scE) \mapsto \tilde{\varphi}(s(\scE))(0)$.
This identification preserves the pairing $Q_A$.

Now we have 
\begin{align}\label{eq:id}
H_{B, \bZ, q_0}^\mathrm{vc} \cong \pi^\ast H_{B, \bZ}^\mathrm{vc} \cong \pi^\ast H_{A, \bZ}^\mathrm{amb} \cong H_{A, \bZ, 0}^\mathrm{amb} \cong H^\bullet_{\psi, \bZ} \lb B, \iota_\ast \bigwedge^\bullet \scT_\bC \rb.
\end{align}
From \pref{th:1} and \pref{lm:5}, it turns out that the last isomorphism of \eqref{eq:id} also preserves the pairing.
Note that the cohomology ring $H^\bullet_\mathrm{amb} \lb Y, \bZ \rb$ is generated by the restrictions to $Y$ of the toric divisors on $X_\Sigmav$.
When we think of the monodromy \eqref{eq:monodromy2} as an automorphism of $H^\bullet_{\psi, \bZ} \lb B, \iota_\ast \bigwedge^\bullet \scT_\bC \rb$ via the identification \pref{eq:id}, it becomes the cup product of
\begin{align}
\exp \lb - 2 \pi \sqrt{-1} c_B \rb,
\end{align}
which coincides with the monodromy of $H_\bZ^\trop$.
Hence, we obtain $(H_{B, \bZ}^\mathrm{vc}, Q_B) \cong (H_\bZ^\mathrm{trop}, Q_\mathrm{trop})$.

Lastly, we compare the Hodge filtrations $\scrF_B$ and $\scrF_\trop$.
Let $\Psi \colon D_\varepsilon \to \check{D}$ be the map induced by \pref{eq:descend} under the current setup.
When we identify $H_{B, \bZ, q_0}^\mathrm{vc}$ with $H_{A, \bZ, 0}^\mathrm{amb}$ via \eqref{eq:id}, the limit Hodge structure $\Psi(0)$ is given by
\begin{align}
\Psi(0)=F_0:= \lc F_0^p:=\bigoplus_{i=0}^{d-p} H^{\leq 2(d-p)}_\mathrm{amb}(Y, \bC) \rc_{p=1}^d.
\end{align}
This filtration $\Psi(0)=F_0$ defines a filtration on
\begin{align}\label{eq:nu1}
\nu \colon \scO_{\lc 0 \rc}^\mathrm{log} \otimes_\bZ H_{B, \bZ}^\mathrm{vc} \xrightarrow{\sim} \scO_{\lc 0\rc}^\mathrm{log} \otimes_\bZ H_{A, \bZ, 0}^\mathrm{amb} =  \scO_{\lc 0\rc}^\mathrm{log} \otimes_\bC H^\bullet_\mathrm{amb} \lb Y, \bC \rb,
\end{align}
where $\nu$ is the restriction of the isomorphism \pref{eq:diskisom} to $\lc 0\rc^\mathrm{log}$.
By the isomorphism $\psi \otimes_\bZ \bC \colon H^\bullet_\mathrm{amb} \lb Y, \bC \rb \to H^\bullet_\psi \lb B, \iota_\ast \bigwedge^\bullet \scT_\bC \rb$, the $\scO_{\lc 0 \rc}^\mathrm{log}$-modules of \eqref{eq:nu1} are isomorphic to 
\begin{align}\label{eq:nu2}
\scO_{\lc 0 \rc}^\mathrm{log} \otimes_\bZ H_\bZ^\trop \xrightarrow{\sim} \scO_{\lc 0\rc}^\mathrm{log} \otimes_\bC H^\bullet_{\psi} \lb B, \iota_\ast \bigwedge^\bullet \scT_\bC \rb.
\end{align}
We also have $(\psi \otimes_\bZ \bC) \lb F_0 \rb =\lc \bigoplus_{i=0}^{d-p} H^i_\mathrm{\psi} \lb B, \iota_\ast \bigwedge^i \scT_\bC \rb \rc_{p=1}^d$.
This defines the filtration $\scrF_\mathrm{trop}$ \eqref{eq:filt} on \eqref{eq:nu2}.
\end{proof}

\begin{remark}\label{rm:res}
Concerning the relation between the monodromy and radiance obstructions, it is known by \cite[Theorem 5.1]{MR2669728} that in Gross--Siebert program, the cup product of the radiance obstruction coincides with the residue of the logarithmic extension of the Gauss--Manin connection under the identification of logarithmic Dolbeault cohomology groups and tropical cohomology groups $H^q \lb B, \iota_\ast \bigwedge^p \scT_\bC^\ast \rb$.
\end{remark}

For $p \in \lc 0, \cdots, d \rc$ and $z \in \lc z \in \bC \relmid \Im(z) > 0 \rc$, we define the subspaces $F^p_\mathrm{trop}(z), H^{p,q}_\mathrm{trop}(z)$ of $H^\bullet_{\psi} \lb B, \iota_\ast \bigwedge^\bullet \scT_\bC \rb$ by
\begin{align}\label{eq:nilp}
F^p_\mathrm{trop}(z)&:= \exp \lb - 2 \pi \sqrt{-1} z \cdot c_B \rb \cdot F^p_\mathrm{trop}, \\ \label{eq:nilp2}
H^{p,q}_\mathrm{trop}(z)&:= F^p_\mathrm{trop}(z) \cap \kappa (F^q_\mathrm{trop}(z)),
\end{align}
where $F_\mathrm{trop}^p:=\bigoplus_{i=0}^{d-p} H^i_\mathrm{\psi} \lb B, \iota_\ast \bigwedge^i \scT_\bC \rb$, $q=d-p$, and $\kappa$ denotes the real involution with respect to the real structure $H^\bullet_{\psi, \bZ} \lb B, \iota_\ast \bigwedge^\bullet \scT_\bC \rb \otimes_\bZ \bR$.
\eqref{eq:nilp} corresponds to the nilpotent orbit of the period map $\Phi$.

\textcolor{red}{}
\begin{remark}\label{rm:rs19}
Let $(\check{B}, \check{\scP}, \check{\varphi})$ be an integral affine manifold with simple singularities equipped with a polyhedral structure and a multi-valued piecewise affine function.
We set $A:=\bC \ld H^1\lb \check{B}, \iota_\ast \scT_\bZ^\ast \rb^\ast \rd$.
For $(\check{B}, \check{\scP}, \check{\varphi})$, we can construct a canonical formal family over $A\ld\ld q \rd\rd$ with central fiber classifying log Calabi--Yau spaces over the standard log point whose intersection complex is $(\check{B}, \check{\scP})$ \cite{MR2846484}, \cite[Theorem A.8]{GHS16}.
The family $\lc V_q \rc_q$ which we consider in this paper is a one-parameter subfamily of this family.

For the canonical family over $A\ld\ld q \rd\rd$, the period integral of the holomorphic volume form over a $d$-cycle $\beta$ constructed from a tropical $1$-cycle $\beta_\trop \in H_1\lb \check{B}, \iota_\ast \scT_\bZ \rb$ is computed in \cite[Theorem 1.7]{RS19}.
Here $H_1\lb \check{B}, \iota_\ast \scT_\bZ \rb$ denotes the sheaf homology of $\iota_\ast \scT_\bZ$.
The relation with this work can be explained as follows:
The result of the period integral contains the factor
\begin{align}\label{eq:factor1}
\log (q) \cdot \la \beta_\trop , c_1(\check{\varphi}) \ra,
\end{align}
where $c_1(\check{\varphi}) \in H^1\lb \check{B}, \iota_\ast \scT_\bZ^\ast \rb$ is the first Chern class of $\check{\varphi}$ (cf. \cite[Definition 1.46]{MR2213573}) and $\la -, - \ra$ is the pairing 
\begin{align}\label{eq:perf}
\la -, - \ra \colon H_1\lb \check{B}, \iota_\ast \scT_\bZ \rb \otimes H^1\lb \check{B}, \iota_\ast \scT_\bZ^\ast \rb \to \bZ.
\end{align}
This pairing is proved to be perfect over $\bQ$ in \cite[Theorem 3]{Rud20}.
Via the discrete Legendre transformation (cf. \cite[Section 1.4]{MR2213573}), this is isomorphic to
\begin{align}\label{eq:perf2}
H^{d-1} \lb B, \iota_\ast \bigwedge^{d-1} \scT_\bZ \rb \otimes H^1\lb B, \iota_\ast \scT_\bZ \rb \to \bZ,
\end{align}
where $B$ is the integral affine manifold with singularities that is discrete Legendre dual to $(\check{B}, \check{\scP}, \check{\varphi})$.
The holomorphic volume form $\Omega$ corresponds to the subspace $H^{d,0}_\mathrm{trop}(z)=\exp \lb - 2 \pi \sqrt{-1} z \cdot c_B \rb$ of \eqref{eq:nilp2}.
The pairing of the $H^1\lb B, \iota_\ast \scT_\bC \rb$-component of $H^{d,0}_\mathrm{trop}(z)$ and an element $\beta_\trop \in H_1\lb \check{B}, \iota_\ast \scT_\bZ \rb \cong H^{d-1} \lb B, \iota_\ast \bigwedge^{d-1} \scT_\bZ \rb$ is
\begin{align}\label{eq:factor2}
\la \beta_\trop, - 2 \pi \sqrt{-1} z \cdot c_B \ra=- 2 \pi \sqrt{-1} z \cdot \la \beta_\trop, c_B \ra.
\end{align}
The first Chern class $c_1(\check{\varphi})$ of $\check{\varphi}$ coincides with the radiance obstruction $c_B$ of $B$ under the isomorphism $H^1\lb \check{B}, \iota_\ast \scT_\bZ^\ast \rb \cong  H^1\lb B, \iota_\ast \scT_\bZ \rb$ \cite[Proposition 1.50.3]{MR2213573}.
Hence, \eqref{eq:factor1} and \eqref{eq:factor2} agree.
The description in \cite[Theorem 1.7]{RS19} is discrete Legendre dual to that in this paper.

\end{remark}

\begin{corollary}
When the imaginary part of $z$ is sufficiently large, one has
\begin{align}\label{eq:HR2}
\lb \sqrt{-1} \rb^{p-q} Q_\trop \lb \phi, \kappa(\phi) \rb > 0
\end{align}
for any element $\phi \in H^{p,q}_\mathrm{trop}(z) \setminus \lc 0 \rc$. 
\end{corollary}
\begin{proof}
From \pref{th:2}, we can see that the tropical period of $B$ is a polarized logarithmic Hodge structure on the standard log point $\lc 0\rc$.
The positivity condition \eqref{eq:positivity} of PLH is equivalent to \pref{eq:HR2}.
\end{proof}

\begin{remark}
When $p=d, q=0$ or $p=0, q=d$, the left hand side of \eqref{eq:HR2} is equal to
\begin{align}
&\lb \sqrt{-1} \rb^{d} Q_\trop \lb \exp \lb -2 \pi \sqrt{-1} z \cdot c_B \rb, \exp \lb -2 \pi \sqrt{-1} \overline{z} \cdot c_B \rb \rb \\
&=\lb \sqrt{-1} \rb^{d} \cdot \lb 2 \pi \sqrt{-1} \rb^{d} \cdot \frac{1}{d!} \cdot \lb 2 \pi \sqrt{-1} \cdot \lb z-\overline{z} \rb \rb^d \bigwedge^d c_B \\
&=\lb 2 \pi \rb^{2d} \cdot \frac{1}{d!} \cdot 2^d \cdot (\Im{z})^d \bigwedge^d c_B.
\end{align}
Hence, the inequality \eqref{eq:HR2} is equivalent to
\begin{align}\label{eq:volume}
\bigwedge^d c_B > 0.
\end{align}
%Recall that $\bigwedge^d c_B$ is equal to the volume of $B$.
%Therefore, \pref{eq:volume} is an inequality that holds obviously.
\end{remark}

\begin{remark}
The inequalities \eqref{eq:HR2} impose constraints on the value of radiance obstructions.
Since the inequality \eqref{eq:HR2} correspond to the positivity condition of PLH, which stems from the Hodge--Riemann bilinear relation, the inequalities \eqref{eq:HR2} can be regarded as a tropical version of the Hodge--Riemann bilinear relations.
On the other hand, we can easily check that the equality \eqref{eq:HR1} does not impose any constraints on the value of radiance obstructions.
Griffiths transversality \eqref{eq:Gtrans} becomes
\begin{align}
- 2 \pi \sqrt{-1} \cdot c_B \cdot F_\mathrm{trop}^p \subset F_\mathrm{trop}^{p-1},
\end{align}
which follows immediately from the definition of $F_\mathrm{trop}^p$.
Hence, Griffiths transversality \eqref{eq:Gtrans} also does not impose any constraints on the value of radiance obstructions.
\end{remark}

\begin{remark}
Consider the case $d=2$.
The space in which radiance obstructions of tropical K3 hypersurfaces take value is
\begin{align}\label{eq:domain}
\lc \sigma \in H^2_\mathrm{amb}(Y, \bZ) \otimes_\bZ \bR \relmid \bigwedge^2 \sigma  >0 \rc.
\end{align}
This can be regarded as the period domain of tropical K3 hypersurfaces.
This is the numerator of the moduli space of lattice polarized tropical K3 surfaces \cite[Section 5]{HU18}.
In \cite{MR3859764}, \cite{OO18}, they construct Gromov--Hausdorff compactifications of polarized complex K3 surfaces by adding moduli spaces of lattice polarized tropical K3 surfaces to their boundaries.
\end{remark}

\bibliographystyle{amsalpha}
\bibliography{bibs}

\noindent
Yuto Yamamoto

Center for Geometry and Physics, Institute for Basic Science (IBS), Pohang 37673, Republic of Korea

{\em e-mail address}\ : \  yuto@ibs.re.kr
\ \vspace{0mm} \\

\end{document}